\documentclass[11pt]{article}

\usepackage{graphicx}
\usepackage{amsmath,amsfonts,amsthm,amssymb}
\usepackage{algpseudocode,algorithm}
\usepackage{fullpage}

\usepackage[colorinlistoftodos,bordercolor=orange,backgroundcolor=orange!20,linecolor=orange,textsize=scriptsize]{todonotes}

\DeclareMathOperator{\dom}{dom}

\newcommand*\rot{\rotatebox{90}}

\usepackage{subfigure}

\newcommand{\R}{{\mathbf R}}

\newcommand{\Exp}{\mathbf{E}}
\newcommand{\Prob}{\mathbf{P}}

\newcommand{\eqdef}{:=}

\newcommand{\xbi}{x^{(i)}}

\newcommand{\lng}{\langle}
\newcommand{\rng}{\rangle}

\newcommand{\h}{h^{(i)}}
\newcommand{\ii}{^{(i)}}
\newcommand{\jj}{^{(j)}}

\newcommand{\lip}{L}

\newcommand{\gfx}{\nabla f(x)}
\newcommand{\gfi}{(\nabla f(x))^{(i)}}
\newcommand{\hshx}{(h(x))_{[\hatS]}}

\newcommand{\hix}{(h(x))^{(i)}}

\newcommand{\gix}{(g(x))^{(i)}}
\newcommand{\fH}{\mathcal{H}_{v}}
\newcommand{\calJ}{\mathcal{J}}
\newcommand{\hatS}{\hat S}
\newcommand{\E}{\mathbb{E}}

\newtheorem{theorem}{Theorem}
\newtheorem{definition}[theorem]{Definition}
\newtheorem{assumption}[theorem]{Assumption}
\newtheorem{lemma}[theorem]{Lemma}

\newtheorem{corollary}[theorem]{Corollary}
\newtheorem{remark}[theorem]{Remark}
\newtheorem{example}[theorem]{Example}

\usepackage{color}

\definecolor{orange}{rgb}{1,0.5,0}

\newcommand{\note}[1]{{\color{orange} #1 }}

\newcommand{\vsubset}[2]{#1_{[#2]}}
\newcommand{\vc}[2]{#1^{(#2)}}

\title{On the Complexity of Parallel Coordinate Descent}

\author{
Rachael Tappenden
  \and
  Martin Tak\'{a}\v{c}\\
  \and
   Peter Richt\'{a}rik
}

\date{}

\begin{document}

\maketitle

\begin{abstract}
In this work we study the parallel  coordinate descent method (PCDM) proposed by Richt\'arik and Tak\'a\v{c} \cite{Richtarik12a} for minimizing a regularized convex function. We adopt elements from the work of Lu and Xiao \cite{Lu13}, and combine them with several new insights, to obtain sharper iteration complexity results for PCDM than those presented in \cite{Richtarik12a}. Moreover, we show that PCDM is monotonic in expectation, which was not confirmed in \cite{Richtarik12a}, and we also derive the first high probability iteration complexity result where the initial levelset is unbounded.
\end{abstract}

\section{Introduction}

Block coordinate descent methods are being thrust into the optimization spotlight because of a dramatic increase in the size of real world problems, and because of the ``Big data'' phenomenon. It is little wonder, when these seemingly simple methods, with low iteration costs and low memory requirements, can solve problems where the dimension is more than one billion, in a matter of hours \cite{Richtarik12a}.

There is an abundance of coordinate descent variants arising in the literature including: \cite{FR:SPCDM,
jaggi2014communication,li2009coordinate,liu2013asynchronous,ma2015adding,necoara2013distributed,
necoara2012efficiency,richtarik2012efficient,
RT:NSync,Saha10,schmidtcoordinate,takac2013mini,
Tao12,TRG:Inexact,Tseng01,Tseng09,Wright12,wrightcoordinate,Wu08}.
The main differences between these methods is the way in which the block of coordinates to update is chosen, and also how the subproblem to determine the update to apply a block of variables is to be solved. The current, state-of-the-art block coordinate descent method is the Parallel (block) Coordinate Descent Method (PCDM) of Richt\'{a}rik and Tak\'a\v{c}~\cite{Richtarik12a}. This method selects the coordinates to update \emph{randomly} and the update is determined by \emph{minimizing an overapproximation of the objective function} at the current point (see Section \ref{S_PCDM} for a detailed description). PCDM can be applied to a problem with a general convex composite objective, it is supported by strong iteration complexity results to guarantee the method's convergence, and it has been tested numerically on a wide range of problems to demonstrate its practical capabilities.

In this work we are interested in the following convex composite/regularized optimization problem
\begin{equation}
\label{D_F}
     \min_{x\in \R^N} F(x) = f(x) + \Psi(x),
\end{equation}
where we assume that $f(x)$ is a continuously differentiable convex function, and $\Psi(x)$ is assumed to be a (possibly nonsmooth) block separable convex regularizer.

The Expected Separable Overapproximation (ESO) assumption introduced in \cite{Richtarik12a} enabled the development of a unified theoretical framework that guarantees convergence of a serial \cite{Richtarik12}, parallel \cite{Richtarik12a} and even distributed \cite{fercoq2014fast,marecek2014distributed,richtarik2013distributed}  version of PCDM.
To benefit from the ESO abstraction, we derive all the results in this paper
based on the assumption that $f$ admits an ESO with respect to a uniform block sampling $\hat{S}$. This concept will be precisely defined in
Section \ref{sec:ESOInAlg}.
For now it is enough to say that updating a random set of $\tau$ coordinates (selected uniformly at random) is one particular uniform sampling and the ESO enables us to overapproximate the expected value of the function at the next iteration by a separable function, which is easy to minimize in parallel.

\subsection{Brief literature review}

Nesterov \cite{Nesterov12} provided some of the earliest
iteration complexity results for a serial Randomized Coordinate Descent Method (RCDM) for problems of the form \eqref{D_F}, where $\Psi\equiv 0$, or is the indicator function for simple bound constraints. Later, this work was generalized to optimization problems with a composite objective of the form \eqref{D_F}, where the function $\Psi$ is any (possibly nonsmooth) convex (block) separable function \cite{Richtarik12,Richtarik12a}.

One of the main advantages of randomized coordinate descent methods is that each iteration is extremely cheap, and can require as little as a few multiplications in some cases \cite{richtarik2012efficient}. However, a large number of iterations may be required to obtain a sufficiently accurate solution, and for this reason, parallelization of coordinate descent methods is essential.
%Due to the fact that each iteration of
%RCDM can be as cheap as a few multiplication \cite{richtarik2012efficient}
%and a many iterations are required to obtain satisfactory accuracy, the parallelization of this method is a must.

The SHOTGUN algorithm presented in \cite{Bradley} represents a na\"ive way of parallelizing RCDM, applied to functions of the form \eqref{D_F} where $\Psi \equiv \|\cdot\|_1$. They also present theoretical results to show that parallelization can lead to algorithm speedup. Unfortunately, their results show that only a small number of coordinates should be updated in parallel at each iteration, otherwise there is no guarantee of algorithm speedup.

%The first true complexity analysis of Parallel RCDM (PCDM) was provided in \cite{Richtarik12a}, where ESO assumption was introduced and ESO parameters were derived for some particular samplings. In \cite{richtarik2013distributed, marecek2014distributed,fercoq2014fast} distributed version of PCDM was considered and the convergence was proved just by deriving the ESO parameters for distributed sampling.
The first true complexity analysis of Parallel RCDM (PCDM) was provided in \cite{Richtarik12a} after the authors developed the concept of an Expected Separable Overapproximation (ESO) assumption, which was central to their convergence analysis. The ESO gives an upper bound on the expected value of the objective function after a parallel update of PCDM has been performed, and depends on both the objective function, and the particular `sampling' (way that the coordinates are chosen) that was used. Moreover, several distributed PCDMs were considered in \cite{fercoq2014fast,marecek2014distributed,richtarik2013distributed} and their convergence was proved simply by deriving the ESO parameters for particular distributed samplings.

In \cite{Fercoq:accelerated,lin2014accelerated} the accelerated PCDM was presented and its efficient distributed implementation was considered in
\cite{fercoq2014fast}. Recently, there has also been a focus on PCDMs that use an arbitrary sampling of coordinates \cite{alpha1, alpha2,quartz,RT:NSync}.

\subsection{Summary of contributions}
In this section we summarize the main contributions of this paper (not in order of significance).
\begin{enumerate}
  \item
  \textbf{No need to enforce ``monotonicity''.}
  PCDM in \cite{Richtarik12a} was analyzed (for a general convex composite function of the form \eqref{D_F}) under a monotonicity assumption; if, at any iteration of PCDM, an update was computed that would lead to a higher objective value than the objective value at the current point, then that update is rejected. Hence, PCDM presented in \cite{Richtarik12a} included a step to force monotonicity of the function values at each iteration. In this paper we confirm that the monotonicity test is redundant, and can be removed from the algorithm.
%  \item \textbf{Sharper iteration complexity results.}
%  In case that for the staring point $x_0$ the
%  levelset $\{x \in \R^N: F(x) \leq F(x_0)\}$
%  is unbounded, the analysis of PCDM in
%   \cite{Richtarik12a}
%   does not prove a convergence.
%   However, in this paper we show that the algorithm converge in expectation
%   to the optimal solution even in the case when the levelset is unbounded.
%For detailed discussion see Section
%\ref{sec:discusion}.
\item \textbf{First high-probability results for PCDM without levelset information.}
Currently, the high probability iteration complexity results for coordinate descent type methods require the levelset to be bounded. In this paper we derive the first high-probability result which \emph{does not rely on the size of the levelset}. In particular, the analysis of PCDM in \cite{Richtarik12a} assumes that the levelset $\{x \in \R^N: F(x) \leq F(x_0)\}$ is bounded for the initial point $x_0$, and under this assumption, convergence is guaranteed. However, in this paper we show that PCDM will converge, in expectation, to the optimal solution even if the levelset is unbounded (see Section~\ref{S_iterationcomplexity}).

  \item \textbf{Sharper iteration complexity results.} In this work we obtain sharper iteration complexity results for PCDM than that those presented in \cite{Richtarik12a} and Table \ref{Table_Comparison_Complexity} summarizes our findings. A thorough discussion of the results can be found in Section \ref{S_Discussion_complexity}. We briefly describe the variables used in the table (all will be properly defined in later sections.) Variable $c$ is a constant, $k$ is the iteration counter, $\alpha\in [0,1]$ is the expected proportion of coordinates updated at each iteration, $\xi_0 = F(x_0)-F_*$, and $v$ is a (vector) parameter of the method. Also, $\mu_f$ and $\mu_\Psi$ are the (strong) convexity constants of $f$ and $\Psi$ respectively (both with respect to $\|\cdot\|_v$ for some $v$) and $\epsilon$ and $\rho$ are the desired accuracy and confidence level respectively. (C=Convex, SC=Strongly Convex).

  %Specifically, Table \ref{Table_Comparison_Rates} compares the convergence rates and Table \ref{Table_Comparison_Complexity} compares the iteration complexity results. We briefly describe the variables used in the two tables. (All will be properly defined in later sections.) Variable $c$ is a constant, $k$ is the iteration counter, $\alpha$ is a probability, $R_{v,0}^2$ is a measure of the distance from the initial point $x_0$ to the solution $x_*$, $\xi_0 = F(x_0)-F_*$, and $\beta$ is a parameter of the method. Also, $\mu_f$ and $\mu_\Psi$ are the (strong) convexity constants of $f$ and $\Psi$ respectively (both with respect to $\|\cdot\|_v$ for some $w$) and $\epsilon$ and $\rho$ are the desired accuracy and confidence level respectively.
      \begin{table}[h!]\centering
{\begin{tabular}{| c| c| c| c|}
\hline
$F$ & Richt\'{a}rik and Tak\'{a}\v{c} \cite{Richtarik12a} & This paper & Theorem \\
\hline
& & & \\[-0.5em]
C & $\displaystyle\frac{2c}{\alpha \epsilon}\left( 1+\log\left(\frac{1}{\rho} \right)\right) + 2 - \frac{2c}{\alpha\xi_0}$ &
$\displaystyle\frac{2 c}{\alpha \epsilon}\left(1 + \log\left(\frac{\frac12
%R_{v,0}^2
\|x_0 - x^*\|^2_v
 + \xi_0}{2c\rho} \right) \right) + 2 - \frac{1}{\alpha}$ & \ref{T_Complexity}(i)\\
& & & \\[-0.5em]
\hline
& & & \\[-0.5em]
SC & $\displaystyle\frac{1 + \mu_\Psi}{\alpha(\mu_f + \mu_\Psi)}\log \left(\frac{\xi_0}{\epsilon \rho} \right)$
& $\displaystyle\frac{1 + \mu_f + 2\mu_\Psi}{2\alpha(\mu_f+\mu_\Psi)}\log\left(\frac{\frac{1+\mu_\Psi}{2}
%R_{v,0}^2
\|x_0 - x^*\|^2_v
+\xi_0}{\epsilon \rho}\right)$ & \ref{T_Complexity}(ii)\\[1.5em]
\hline
\end{tabular}}
\caption{Comparison of the iteration complexity results for PCDM obtained in \cite{Richtarik12a} and in this paper. The analysis used in this paper provides a sharper iteration complexity result in both the convex and strongly convex cases when $\epsilon$ and/or $\rho$ are small.}
\label{Table_Comparison_Complexity}
\end{table}

\item \textbf{Improved convergence rates for PCDM.} In this work we show that PCDM converges at a faster rate than that given in \cite{Richtarik12a}, in both the convex and strongly convex cases. Table \ref{Table_Comparison_Rates} provides a summary of our results and a thorough discussion can be found in Section \ref{S_comparisonrate}.
    %\todo[inline]{Write the discussion/comparison.}
\begin{table}[h!]\centering
{\begin{tabular}{| c| c| c| c|}
\hline
$F$ & Richt\'{a}rik and Tak\'{a}\v{c} \cite{Richtarik12a} & This paper & Theorem \\
\hline
& & &\\[-0.5em]
C & $\displaystyle \frac{2c\xi_0}{2c + \alpha k \xi_0}$ & $\displaystyle\frac{1}{1+\alpha k} \left(\frac12 \|x_0-x^*\|_v^2 + \xi_0 \right)$ & \ref{T_convergence_rate}(i)\\
& & &\\[-0.5em]
\hline
& & &\\[-0.5em]
SC & $\displaystyle\left( 1 - \alpha \frac{\mu_f + \mu_\Psi}{1 + \mu_\Psi} \right)^k\xi_0$
& $\displaystyle\left(1 - \frac{2\alpha (\mu_f + \mu_\Psi)}{1+\mu_f + 2\mu_\Psi}\right)^k \left(\frac{1+\mu_\Psi}{2}
%R_{v,0}^2
\|x_0 - x^*\|^2_v
+ \xi_0 \right)$ & \ref{T_convergence_rate}(ii)\\[1.5em]
\hline
\end{tabular}}
\caption{Comparison of the convergence rates for PCDM obtained in \cite{Richtarik12a} and in this paper. (C=Convex, SC=Strongly Convex). The analysis used in this paper provides a better rate of of convergence in both the convex and strongly convex cases when $\epsilon$ and/or $\rho$ are small.}
\label{Table_Comparison_Rates}
\end{table}

\end{enumerate}

\subsection{Paper outline}

The remainder of this paper is structured as follows. In Section \ref{Section_Preliminaries} we introduce the notation and assumptions that will be used throughout the paper. Section \ref{S_PCDM} describes PCDM of Richt\'arik and Tak\'a\v c \cite{Richtarik12a} in detail. We also present a new convergence rate result for PCDM, which is sharper than that presented in \cite{Richtarik12a}. The proof of the result is given in Section \ref{sec:proof} along with several necessary technical lemmas.

In Section \ref{S_iterationcomplexity} we present several iteration complexity results, which show that PCDM will converge to an $\epsilon$-optimal solution with high probability. In Section \ref{sec:HPR:Case1} we provide the first iteration complexity result for PCDM that does not require the assumption of a bounded levelset. The results shows that PCDM requires $\mathcal{O}(\frac{1}{\rho})$ iterations, so we have devised a `multiple run strategy' that achieves the classical $\mathcal{O}(\log \frac{1}{\rho})$ result. Moreover, in Section \ref{sec:HPR:Case1} we present a high probability iteration complexity result for PCDM, that assumes boundedness of the levelset, which is sharper than the result given in \cite{Richtarik12a}.

In Section \ref{sec:discusion} we give a comparison of the results derived in this work, with the results given in \cite{Richtarik12a}. Then, we present several numerical experiments in Section~\ref{sec:numerical} to highlight the practical capabilities of PCDM under different ESO assumptions. The ESO assumptions are given in Appendix \ref{sec:ESO}, where we also provide a new ESO for doubly uniform samplings (see Theorem \ref{thm:NewESOforDUS}).

\section{Notation and assumptions}
%\section{Block structure }
\label{Section_Preliminaries}

In this section we introduce block structure and associated objects such as norms and projections. The parallel (block) coordinate descent method will operate on blocks instead of coordinates.
%Subsequently, we review existing  assumptions on the smoothness of $f$.
%\todo[inline]{Update text below}
%Our iteration complexity results hold for all of these  smoothness assumptions.

\subsection{Block structure}
\label{S_Block_structure}
The problem under consideration is assumed to have block structure and this is modelled by decomposing the space $\R^N$ into $n$ subspaces as follows. Let $U \in \R^{N \times  N}$ be a column permutation of the $N \times N$ identity matrix and further let $U = [U_1,U_2,\dots,U_n]$ be a decomposition of $U$ into $n$ submatrices, where $U_i$ is $N \times N_i$ and $\sum_{i=1}^n N_i = N$. Note that $U_i^TU_j = I_{N_i}$ when $i = j$ and $U_i^TU_j = \mathbf{0}$ (where $\mathbf{0}$ is the $N_i \times N_j$ matrix of all zeros) when $i \neq j$. Subsequently, any vector $x \in \R^N$ can be written uniquely as
\begin{equation}
\label{D_xdecomp}
x = \sum_{i=1}^n U_ix^{(i)}
\end{equation}
where $x^{(i)} = U_i^Tx \in \R^{N_i}$. For simplicity we will write $x = (x^{(1)},x^{(2)},\dots,x^{(n)})^T$.

In what follows let $\langle \cdot,\cdot \rangle$ denote the standard Euclidean inner product. Then we have
\begin{equation}
\label{D_innerprod}
  \langle x,y \rangle = \left\langle \sum_{i=1}^nU_ix\ii,\sum_{j=1}^nU_jy\jj \right\rangle = \sum_{i=1}^n\sum_{j=1}^n\langle U_j^TU_ix\ii,y\jj \rangle \equiv \sum_{i=1}^n\langle x\ii,y\ii \rangle.
\end{equation}

\paragraph{Norms.}

Further we equip $\R^{N_i}$ with a pair of conjugate Euclidean norms:
\begin{eqnarray}
\label{D_norm_Bi}
  \|h\|_{(i)} \eqdef \langle B_ih,h \rangle^{\frac{1}{2}}, \qquad \|h\|_{(i)}^* = \langle B_i^{-1}h,h \rangle^{\frac{1}{2}}, \qquad h \in \R^{N_i},
\end{eqnarray}
where $B_i \in \R^{N_i \times N_i}$ is a positive definite matrix. For fixed positive scalars $v_1,v_2,\dots,v_n$, let $v = (v_1,\dots,v_n)^T$ and define a pair of conjugate norms in $\R^N$ by
\begin{equation}
\label{D_norm_v}
     \|x\|_v^2 \eqdef  \sum_{i=1}^n v_i \|x^{(i)}\|^2_{(i)}, \quad   (\|y\|_v^*)^2 \eqdef \max_{\|x\|_v \leq 1} \langle y,x \rangle ^2 = \sum_{i=1}^n \frac{1}{v_i} (\|y^{(i)}\|^*_{(i)})^2.
\end{equation}

\paragraph{Projection onto a set of blocks.}

Let $\emptyset \neq S \subseteq \{1,2,\dots,n\}$. Then for $x \in \R^N$ we write
\begin{equation}
\vsubset{x}{S} \eqdef \sum_{i \in S}U_i \xbi,
\end{equation}
and we define $x_{[\emptyset]} \equiv 0$. That is, given $x \in \R^N$, $\vsubset{x}{S}$ is the vector in $\R^N$ whose blocks $i\in S$ are identical to those of $x$, but whose other blocks are zeroed out.

\subsection{Assumptions and strong convexity}

Throughout this paper we make the following assumption regarding the block separability of the function $\Psi$.
\begin{assumption}[Block separability]
\label{SS_separability}
The nonsmooth function $\Psi:\R^N \to \R\cup \{+\infty\}$ is assumed to be block separable, i.e., it can be decomposed as:
\begin{equation}
     \label{D_Psi_sep}
\Psi(x) = \sum_{i=1}^n \Psi_i (x\ii),
\end{equation}
where the functions $\Psi_i: \R^{N_i} \to \R\cup \{+\infty\}$ are
 proper, closed and convex.

\end{assumption}

In some of the results presented in this work we assume that $F$ is strongly convex and we denote the (strong) convexity parameter of $F$, with respect to the norm $\| \cdot \|_v$ for some $v\in \R^n_{++}$, by $\mu_F >0$.  A function $\phi: \R^N \to \R \cup \{+ \infty\}$ is strongly convex with respect to the norm $\| \cdot \|_v$ with convexity parameter $\mu_{\phi}  \geq 0$ if for all $x,y \in \dom \phi$,
\begin{equation}
\label{strongly_convex_1}
     \phi(y) \geq \phi(x) + \langle \phi^{\prime}(x),y-x \rangle + \frac{\mu_{\phi} }{2}\|y-x\|_v^2,
\end{equation}
where $\phi^{\prime}$ is any subgradient of $\phi$ at $x$. The case with $\mu_{\phi}= 0$ reduces to convexity.

Strong convexity of $F$ may come from $f$ or $\Psi$ or both and we will write $\mu_f$ (resp. $\mu_{\Psi}$) for the strong convexity parameter of $f$ (resp. $\Psi$). Following from \eqref{strongly_convex_1}
\begin{equation}
     \label{strongly_convex_4}
\mu_F \geq \mu_f + \mu_\Psi.
\end{equation}

From the first order optimality conditions for \eqref{D_F} we obtain $\langle F^{\prime}(x_*),x-x_* \rangle \geq 0$ for all $x \in$ dom$F$. Combining this with \eqref{strongly_convex_1} used with $y = x$ and $x = x_*$, yields the standard inequality
\begin{equation}
\label{strongly_convex_2}
   F(x) - F_* \geq \frac{\mu_F }{2} \|x-x_*\|_{v}^2,  \qquad x \in \rm{dom} F.
\end{equation}
%Also, using \eqref{S2_upperbound} and \eqref{strongly_convex_1} it can be shown that $\mu_f(L)\leq 1$.

\section{Parallel coordinate descent method}
\label{S_PCDM}
In this section we describe the Parallel Coordinate Descent Method (Algorithm \ref{PCDM}) of Richt\'arik and Tak\'a\v c \cite{Richtarik12a}. We now present the algorithm, and a detailed discussion will follow.
\begin{algorithm}[H]
\caption{PCDM: Parallel Coordinate Descent Method \cite{Richtarik12a} }\label{PCDM}
\begin{algorithmic}[1]
\State choose initial point $x_0 \in\R^N$
\For{$k = 0,1,2,\dots$}
\State randomly choose set of blocks $S_k \subseteq \{1,\dots,n\}$
\label{alg:line:sampling}
  \For{$i \in S_k$ ({\bf in parallel})}
  \State  compute
  %\begin{equation}
%\label{E_argminhi}
$\vc{h(x_k)}{i} = \arg \min_{t \in \R^{N_i}} \Big\{\langle \vc{(\nabla f(x_k))}{i},t \rangle + \frac{  v_i}{2} \|t\|_{(i)}^2 + \Psi_i(\vc{x_k}{i}+t)\Big\}
$%\end{equation}
\label{alg:line:stepsize}
  \EndFor
\State apply the update: $x_{k+1} \gets x_k + \sum_{i\in S_k} U_i \vc{h(x_k)}{i}$
\EndFor
\end{algorithmic}
\end{algorithm}

The algorithm can be described as follows. At iteration $k$ of Algorithm \ref{PCDM}, a set of blocks $S_k$ is chosen, corresponding to the (blocks of) coordinates that are to be updated. The set of blocks is selected via a \emph{sampling}, which is described in detail in Section \ref{S_sampling}. Then, in Steps 4--6, the updates $h(x_k)\ii$, for all $i \in S_k$, are computed \emph{in parallel}, via a small/low dimensional minimization subproblem. (In Section \ref{sec:ESOInAlg}, we describe the origin of this subproblem via an ESO.) Finally, in Step 7, the updates $h(x_k)\ii$ are applied to the current point $x_k$, to give the new point $x_{k+1}$. Notice that Algorithm \ref{PCDM} \emph{does not require knowledge of objective function values.}

We now describe the key steps of Algorithm \ref{PCDM} (Steps 3 and 4--6) in more detail.

\subsection{Step \ref{alg:line:sampling}: Sampling}
\label{S_sampling}
%\todo{Rachael,   this will need some "Englishfying ...:)"}
At the $k$th iteration of Algorithm \ref{PCDM}, a set of indices $S_k\subseteq \{1,\dots,n\}$ (corresponding to the blocks of $x_k$ to be updated) is selected.
Here we briefly explain several schemes for choosing the set of indices $S_k$; a thorough description can be found in \cite{Richtarik12a}. Formally, $S_k$ is a realisation of a \emph{random set-valued mapping} $\hatS$ with values in $2^{\{1,\dots,n\}}$. Richt\'arik and Tak\'a\v c \cite{Richtarik12a} have coined the term \emph{sampling} in reference to $\hatS$.

In what follows, we will assume that all samplings are \emph{proper}. That is, we assume that $p_i>0$ for all blocks $i$, where $p_i$ is the probability that the $i$th block of $x$ is updated.

We state several sampling schemes now.
\begin{enumerate}
%\item \textbf{Uniform:} A sampling $\hatS$ is uniform if all blocks get updated with the same probability.
\item \textbf{Uniform:} A sampling $\hatS$ is uniform if all blocks have the same probability of being updated.
\item \textbf{Doubly uniform:} A doubly uniform sampling is one that generates all sets of equal cardinality with equal probability. That is $\Prob(S') = \Prob(S'')$ whenever $|S'|=|S''|$.
\item \textbf{Nonoverlapping uniform:} A nonoverlapping uniform sampling is one that is uniform and assigns positive probabilities only to sets forming a partition of $\{1,\dots,n\}$.
\end{enumerate}
In fact, doubly uniform and nonoverlapping uniform samplings are special cases of uniform samplings, so in this work all results are proved for uniform samplings. Other samplings, which are also special cases of uniform samplings, are presented in \cite{Richtarik12a}, but we omit details of all, except a $\tau$-nice sampling, for brevity. %Let us highlight one particular sampling, namely, a $\tau$-nice sampling.
We say that a sampling $\hatS$ is $\tau$-nice, if for any $S \subseteq \{1,2,\dots,n\}$ we have
\begin{equation}
 \Prob(\hatS = S)=\begin{cases}
                     0, &\mbox{if} \ |S|\neq\tau,\\
                     \frac{\tau! (n-\tau)!}{n!}
                     ,&\mbox{otherwise}.
                    \end{cases}
\end{equation}

\subsection{Step
\ref{alg:line:stepsize}: Computing the step-length} \label{sec:ESOInAlg}
%The step size is chosen in such a way that it minimize the
%expected upper-bound of the next iterate, where expectation is considered with respect to the sampling $\hatS$ used.
The block update $h(x_k)\ii$ is chosen in such a way that an upper bound on the expected function value at the next iterate is minimized, with respect to the particular sampling $\hatS$ that is used.
The construction of the expected upper bound should be (block) separable to ensure efficient parallelizability.
Before we focus on how to construct the expected upper-bound on $F$ we
will state a definition of ESO.
\begin{definition}[Expected Separable Overapproximation; Definition~5 in \cite{Richtarik12a}]\label{Def_ESO}
Let $v\in \R_{++}^n$
and  $\hatS$ be a proper uniform sampling. We say that $f:\R^N \to \R$ admits an ESO with respect to the sampling $\hatS$ with parameter  $v$, if, for all $x, h \in \R^N$ the following inequality holds:
\begin{equation}
\label{eq:ESOdef}
\Exp[f(x+\vsubset{h}{\hatS}) ] \leq f(x) + \frac{\Exp[|\hatS|]}{n} \left(\lng \nabla f(x),h\rng + \frac{1}{2}\|h\|_v^2\right).
\end{equation}
We say that the ESO is \emph{monotonic} if $\forall S \in \hatS$ such that $ \Prob(S=\hatS) > 0$ the following holds:
\begin{equation*}
 f(x+\vsubset{h}{S})   \leq f(x).
\end{equation*}
\end{definition}
%Now, if we assume that we have a vector $v$ such that \eqref{eq:ESOdef} holds (see Section \ref{sec:ESO} for a review of different smoothness assumptions on $f$ and corresponding ESO parameters $v$ when a doubly uniform sampling is used). Hence, from now on, \emph{we assume that $v$ is the ESO parameter and $\hatS$ is  a proper uniform sampling.} Then
In Appendix \ref{sec:ESO}, a review of different smoothness assumptions on $f$ and corresponding ESO parameters $v$ for a doubly uniform sampling, is given. In all that follows, we assume that $f$ admits an ESO, and that $v$ is the ESO parameter and $\hatS$ is a proper uniform sampling. Then
\begin{align}
\nonumber
\Exp[F(x+\vsubset{h}{\hatS})]
\overset{\eqref{D_F}}{=}&
\Exp[f(x+\vsubset{h}{\hatS})
]
+
\Exp[\Psi(x+\vsubset{h}{\hatS})
]
\\
\overset{\eqref{eq:ESOdef}
\eqref{eq:L_blockseparable}
}{\leq}&
f(x) + \tfrac{\Exp[|\hatS|]}{n} \left(\lng \nabla f(x),h\rng + \tfrac{1}{2}\|h\|_v^2\right)
+
\left(1-\tfrac{\Exp[|\hatS|]}{n}\right) \Psi(x)
+
\tfrac{\Exp[|\hatS|]}{n} \Psi(x+h),
\label{eq:safvfeavfwafcwa}
\end{align}
where we have used that fact that $\Psi$ is block separable and that $\hatS$ is a proper uniform sampling (see \cite[Theorem~4]{Richtarik12a}).

Now, it is easy to see that minimizing the right hand side of \eqref{eq:safvfeavfwafcwa} in $h$ is the same as minimizing the function $\fH$ in $h$, where $\fH$ is defined to be
\begin{equation}
\label{Def_H}
\fH(x,h) \eqdef f(x) + \langle \nabla f(x),h \rangle + \frac{1}{2} \|h\|_v^2 + \Psi(x+h).
\end{equation}
In view of \eqref{D_xdecomp}, \eqref{D_norm_v}, and \eqref{D_Psi_sep}, we can write
\begin{equation*}
\fH(x,h) \eqdef f(x) + \sum_{i=1}^n\Big\{\langle \gfi,\h \rangle + \frac{v_i}{2} \|\h\|_{(i)}^2 + \Psi_i(\xbi+\h)\Big\}.
\end{equation*}
Further, we define
\begin{equation}\label{E_argminhi}
h(x) \eqdef \arg \min_{h \in \R^N}  \fH(x,h),
\end{equation}
which is the update used in Algorithm \ref{PCDM}. Notice that the algorithm \emph{never evaluates function values.}

\subsection{Complexity of PCDM}
%In this section we state the main result of this work, which provides an improved iteration complexity result for PCDM. The proof is provided in Section~\ref{sec:proof}.

%We are now ready to present one of our main results, which is a generalization of Theorem~1 in \cite{Lu13}, and provides an improved iteration complexity result for PCDM. The proof is provided in Section~\ref{sec:proof}. Let us mention that a similar result was given independently\footnote{A preliminary version of this paper was ready in August 2013.} in \cite{necoara2013distributed}, but that result \emph{only holds for the particular ESO described in Theorem~\ref{thm:IonESO}}. However, even for that ESO, our result (Theorem \ref{T_convergence_rate}) is still much better because it depends on $\|x_0-x_*\|_v$ and not on the size of a initial levelset (which could even be unbounded).

We are now ready to present one of our main results, which is a generalization of Theorem~1 in \cite{Lu13}. The result shows that PCDM converges in expectation and provides an sharper convergence rate than that given in \cite{Richtarik12a}. The proof is provided in Section~\ref{sec:proof}. Let us mention that a similar result was given independently\footnote{A preliminary version of this paper was ready in August 2013.} in \cite{necoara2013distributed}, but that result \emph{only holds for the particular ESO described in Theorem~\ref{thm:IonESO}}. However, even for that ESO, our result (Theorem \ref{T_convergence_rate}) is still much better because it depends on $\|x_0-x_*\|_v$ and not on the size of the initial levelset (which could even be unbounded). We state our result now.
\begin{theorem}
\label{T_convergence_rate}
Let $F^*$ be the optimal value of problem \eqref{D_F}, and let $\{x_k\}_{k\geq 0}$ be the sequence of iterates generated by PCDM using a uniform sampling $\hat{S}$. Let $\alpha = \tfrac{\Exp[|\hat{S}|]}{n}$
and suppose that $f$ admits an ESO with respect to the sampling $\hatS$ with parameter $v$. Then for any $k \geq 0$,
\begin{itemize}
  \item[(i)] the iterate $x_k$ satisfies
  \begin{equation}
\label{T_complexityC}
\Exp[F(x_{k})-F_*] \leq \frac{1}{1+\alpha k} \left(\frac{1}{2}\|x_0-x_*\|_v^2 + F(x_{0}) - F_* \right),
\end{equation}
\item[(ii)] if $\mu_f + \mu_\Psi >0$, then the iterate $x_k$ satisfies
\begin{equation}
\label{T_complexitySC}
\Exp[F(x_k) - F_*] \leq \left(1-\frac{2\alpha (\mu_f  + \mu_\Psi )}{1+\mu_f  + 2\mu_\Psi }\right)^k \left(\frac{1 +\mu_\Psi }{2} \|x_0-x_*\|_v^2 + F(x_{0}) - F_* \right).
\end{equation}

\end{itemize}
\end{theorem}

\begin{remark}
  Notice that Theorem \ref{T_convergence_rate} is a \emph{general} result, in the sense that \emph{any ESO can be used for PCDM} and the result holds.
\end{remark}
%%%%%%%%%%%%
\section{Proof of the main result}
\label{sec:proof}
In this section we provide a proof of our main convergence rate result, Theorem \ref{T_convergence_rate}. However, first we will present several preliminary results, including the idea of a composite gradient mapping, and other technical lemmas.

\subsection{Block composite gradient mapping}
We now define the concept of a \emph{block composite gradient mapping} \cite{nesterov2007gradient, Lu13}.
By the first-order optimality conditions for problem \eqref{E_argminhi}, there exists a subgradient $s\ii \in \partial \Psi_i(x\ii + \hix)$ (where $\partial \Psi_i(\cdot)$ denotes the subdifferential of $\Psi_i(\cdot)$) such that
\begin{equation}\label{eq:sjd7d9d}
\gfi +  v_i B_i \hix + s\ii = 0.
\end{equation}
We define the block composite gradient mappings as
\begin{equation}
\label{D_gi}
\gix \eqdef -v_i B_i \hix, \qquad i = 1,\dots,n.
\end{equation}
From \eqref{eq:sjd7d9d} and \eqref{D_gi} we obtain
\begin{equation} \label{eq:shdud9d9}
-\gfi + \gix \in \partial \Psi_i(\xbi + \hix),\qquad i = 1,\dots,n.
\end{equation}
If we let
$\displaystyle
g(x) \eqdef \sum_{i=1}^n U_i \gix$ (compare \eqref{D_xdecomp} and \eqref{D_gi}),
then since $\Psi$ is separable,  \eqref{eq:shdud9d9} can be written as
\begin{equation}
\label{E_subgradfg}
-\gfx + g(x) \in \partial \Psi(x + h(x)).
\end{equation}
Moreover
\begin{equation}
\label{E_hgnorm_equiv}
  \|h(x)\|_v^2 \overset{\eqref{D_norm_v}}{=}  \sum_{i=1}^n v_i  \|\hix\|_{(i)}^2 \overset{\eqref{D_gi}}{=} \sum_{i=1}^n \frac{1}{v_i}  \|B_i^{-1}\gix\|_{(i)}^2 \overset{\eqref{D_norm_Bi}+\eqref{D_norm_v}}{=}  (\|g(x)\|_v^*)^2,
\end{equation}
and
\begin{equation}
\label{E_CS}
\lng g(x),h(x)\rng \overset{\eqref{D_innerprod}+\eqref{D_gi}}{=} -  \|h(x)\|_v^2 \overset{\eqref{E_hgnorm_equiv}}{=} - (\|g(x)\|_v^*)^2.
\end{equation}

%\begin{lemma}
%\label{L_upH}
%Suppose $x,h \in \R^N$. Let $\hatS$ be a proper uniform sampling and let $\alpha = \frac{\Exp[|\hatS|]}{n}$. Then
%\begin{equation}
%\Exp[F(x + \hshx)] - F(x) \leq \alpha \left(\fH(x,h(x)) - F(x)\right).
%\end{equation}
%\end{lemma}

Finally, note that using \eqref{D_norm_Bi}, \eqref{D_norm_v},  \eqref{D_gi} and \eqref{E_hgnorm_equiv}, we get
\begin{equation}\label{eq:sjasujdks}
\| x+ h(x) -y \|_v^2 \; \; = \;\; \|x-y\|_v^2 + 2\langle g(x), y-x \rangle +  \left( \|g(x)\|_v^* \right)^2.
\end{equation}

\subsection{Main technical lemmas}

The following result   concerns the expected value of a block-separable function when a random subset of coordinates is updated.
\begin{lemma}[Theorem 4 in \cite{Richtarik12a}]
\label{L_blockseparable}
Suppose that $\Psi(x) = \sum_{i=1}^n \Psi_i(x\ii)$. For any $x,h \in \R^N$, if we choose a uniform sampling $\hatS$, then letting  $\alpha = \frac{\Exp[|\hatS|]}{n}$, we have
\begin{equation}
\label{eq:L_blockseparable}
\Exp[\Psi(x+\hshx)] = \alpha \Psi(x+h(x)) + (1-\alpha)\Psi(x).
\end{equation}
\end{lemma}

The following technical lemma plays a central role in our analysis. The result can be viewed as a generalization of  Lemma 3 in \cite{Lu13}, which considers the serial case ($\alpha=1$), to the parallel setting.

\begin{lemma}
\label{L_3}
Let $x \in \dom F$ and $x_+ = x + (h(x))_{[\hat{S}]}$, where $\hat{S}$ is any uniform sampling. Then for any $y \in \dom F$,
\begin{eqnarray}
\Exp\left [F(x_+) + \tfrac{\mu_{\Psi} +1}{2}\|x_+ - y \|_v^2\right] &\leq & F(x) + \tfrac{\mu_{\Psi} +1}{2}\|x-y\|_v^2 \notag \\
&& \; - \alpha \left(F(x)-F(y) + \tfrac{\mu_f +\mu_\Psi }{2}\|x-y\|_v^2\right).\label{eq:sjs8sjs8}
\end{eqnarray}
Moreover,
\begin{enumerate}
\item[(i)]
 \begin{equation}
  \label{C_monodecrease}
\Exp\left[F(x_+)\right] \leq F(x) - \frac{\alpha}{2} (\mu_\Psi + 1)\|h(x)\|_v^2 = F(x)- \frac{\alpha}{2 } (\mu_\Psi + 1)(\|g(x)\|_v^*)^2,
\end{equation}
\item[(ii)]
\begin{eqnarray}
\Exp\left [F(x_+) + \tfrac{1}{2}\|x_+ - y \|_v^2\right] &\leq & F(x) + \tfrac{1}{2}\|x-y\|_v^2 - \alpha \left(F(x)-F(y)\right).\label{C_trivial}
\end{eqnarray}

\end{enumerate}
\end{lemma}

\begin{proof}
We first note that
\begin{equation}\label{eq:iuweiuhd00}
\Exp\left[\|x_+ - y\|_v^2\right] \;\; = \;\; \alpha \|x+h(x)-y\|_v^2 + (1-\alpha)\|x-y\|_v^2.
\end{equation}
This is a special case of the identity $\Exp[\psi(u+h_{[\hat{S}]})] = \alpha \psi(u+h) + (1-\alpha)\psi(u)$ (see Lemma \ref{L_blockseparable}, which holds for block separable functions $\psi$),  with $\psi(u) = \|u\|_v^2$, $u=x-y$ and $h=h(x)$.

Further,  for any $h$ for which $x+h \in \dom \Psi$, we have
\begin{equation}\label{eq:jdyugsd87ds}
\Exp[F(x + \vsubset{h}{\hatS} )]
\overset{\eqref{eq:L_blockseparable}
}{\leq} (1-\alpha)F(x) + \alpha \fH(x,h).
\end{equation}
This was established in  \cite[Section 5]{Richtarik12a}.  The claim now follows by combining \eqref{eq:jdyugsd87ds}, used with $h=h(x)$, and the following estimate of $\fH(x,h(x))$:
\begin{eqnarray*}
\fH(x,h(x)) &\overset{\eqref{Def_H}}{=}& f(x) + \langle \nabla f(x),h(x) \rangle + \tfrac{1}{2} \|h(x)\|_v^2 + \Psi(x+h(x))
\\
&\overset{\eqref{strongly_convex_1}+\eqref{E_subgradfg}}{\leq}& f(y) + \lng \nabla f(x), x-y\rng - \tfrac{\mu_f }{2}\|y-x\|_v^2+ \langle \nabla f(x),h(x) \rangle + \tfrac{1}{2} \|h(x)\|_v^2
\\
&\phantom{\leq}&  + \;\Psi(y) + \lng -\nabla f(x) + g(x),x+h(x)-y \rng - \tfrac{\mu_\Psi}{2}\|x+h(x)-y\|_v^2\\
&=& F(y) + \lng g(x) ,x-y\rng + \lng g(x),h(x)\rng - \tfrac{\mu_f}{2}\|y-x\|_v^2\\
&\phantom{\leq}& - \tfrac{\mu_\Psi}{2}\|x+h(x)-y\|_v^2 + \;\tfrac{1}{2} \|h(x)\|_v^2\\
&\overset{\eqref{E_CS}}{=}& F(y) + \lng g(x) ,x-y\rng  - \tfrac{\mu_f}{2}\|y-x\|_v^2 - \tfrac{\mu_\Psi}{2}\|x+h(x)-y\|_v^2  - \tfrac{1}{2} (\|g(x)\|_v^*)^2\\
&\overset{\eqref{eq:sjasujdks}}{=}&
F(y)  + \tfrac{1-\mu_f}{2}\|y-x\|_v^2 - \tfrac{\mu_\Psi+1}{2}\|x+h(x)-y\|_v^2\\
&\overset{\eqref{eq:iuweiuhd00}}{=}&
F(y)  + \tfrac{1-\mu_f}{2}\|y-x\|_v^2 - \tfrac{\mu_\Psi+1}{2\alpha}\left(\Exp\left[\|x_+ - y\|_v^2 \right] - (1-\alpha)\|x-y\|_v^2\right).
\end{eqnarray*}
%\todo[inline]{Simplify to $\mu_f, \mu_\Psi$ EVERYWHERE IN THE PAPER.}
%Above we have used simplified notation: $\mu_f=\mu_f(w)$ and $\mu_\Psi=\mu_\Psi(w)$.

Part (i) follows by  letting $x=y$  and using \eqref{eq:iuweiuhd00} and \eqref{E_CS}. Part (ii) follows as a special case by choosing $\mu_f=\mu_\Psi=0$.
\end{proof}

Property (i) means that  function values $F(x_k)$ of PCDM are monotonically decreasing in expectation when conditioned on the previous iteration.

\subsection{Proof of Theorem~\ref{T_convergence_rate}}

\begin{proof}
Let $x_*$ be an arbitrary optimal solution of \eqref{D_F}. Let $r_k^2  = \|x_k - x_*\|_v^2$,  $g_k = g(x_k)$, $h_k = h(x_k)$ and $F_k = F(x_k)$. Notice that $x_{k+1} = x_{k} + (h_k)_{[S_{k}]}$. By subtracting $F_*$ from both sides of  \eqref{C_trivial}, we get
\begin{equation*}
\Exp\Big[\tfrac{1}{2}r_{k+1}^2 + F_{k+1} - F_*\;|\;x_k\Big] \leq \left(\tfrac{1}{2}r_k^2 + F_k -F_*\right) -\alpha(F_k - F_*),
\end{equation*}
and taking expectations with respect to the whole history of realizations of $S_l, l\leq k$ gives us
\begin{equation*}
\Exp\Big[\tfrac{1}{2}r_{k+1}^2 + F_{k+1} - F_*\Big] \leq \Exp\Big[\tfrac{1}{2}r_k^2 + F_k -F_*\Big] -\alpha \Exp\big[F_k - F_*\Big].
\end{equation*}
Applying this inequality recursively and using the fact that $\Exp[F_j]$ is monotonically decreasing for $j=0,1,\dots,k+1$ \eqref{C_monodecrease}, we obtain
\begin{align*}
\Exp[F_{k+1}-F_*] &\leq  \Exp\Big[\tfrac{1}{2}r_{k+1}^2 + F_{k+1} - F_* \Big] \leq  \tfrac{1}{2}r_{0}^2 + F_0 - F_* - \alpha \sum_{j=0}^k (\Exp[F_j]-F_*)\\
&\leq  \tfrac{1}{2}r_{0}^2 + F_0 - F_* - \alpha (k+1) (\Exp[F_{k+1}]-F_*),
\end{align*}
which leads to \eqref{T_complexityC}.

We now prove \eqref{T_complexitySC} under the strong convexity assumption $\mu_f+\mu_\Psi >0$. From \eqref{eq:sjs8sjs8} we get
\begin{eqnarray}
\Exp\Big[\tfrac{1 + \mu_\Psi}{2}r_{k+1}^2 + F_{k+1} - F_* \;|\; x_k\Big] &\leq& \left(\tfrac{1 + \mu_\Psi}{2}r_k^2 + F_k - F_*  \right)  - \alpha \left(\tfrac{\mu_f + \mu_\Psi}{2}r_k^2 + F_k - F_*  \right).
\label{E_1}
\end{eqnarray}
Notice that for any  $0 \leq \gamma \leq 1$ we have
\begin{eqnarray*}
\tfrac{\mu_f + \mu_\Psi}{2}r_k^2 + F_k - F_*
&=&
\gamma (\tfrac{\mu_f + \mu_\Psi}{2}r_k^2 + F_k - F_*)
+(1-\gamma) (\tfrac{\mu_f + \mu_\Psi}{2}r_k^2 + F_k - F_*)
\\
 &\overset{\eqref{strongly_convex_4}+\eqref{strongly_convex_2}}{\geq}& \gamma\left(\tfrac{\mu_f + \mu_\Psi}{2}r_k^2 + F_k - F_* \right) + (1-\gamma)(\mu_f + \mu_\Psi)r_k^2.
 \end{eqnarray*}
Choosing  \begin{equation}\label{D_gamma}
 \gamma^* \eqdef  \frac{2(\mu_f +\mu_\Psi )}{1 +  \mu_f  + 2\mu_\Psi } \in [0,1]
 \end{equation}
 we obtain
 \begin{eqnarray*}
 \tfrac{\mu_f + \mu_\Psi}{2}r_k^2 + F_k - F_*
&\overset{\eqref{D_gamma}}{\geq}& \gamma^* \left(\tfrac{1 + \mu_\Psi}{2}r_k^2 + F_k - F_*\right).
\end{eqnarray*}
Combining the inequality above with \eqref{E_1} gives
\begin{equation}
\label{E_rkxik_recursion}
\Exp\Big[\tfrac{1 + \mu_\Psi}{2}r_{k+1}^2 + F_{k+1} - F_* \;|\; x_k\Big] \leq(1-\gamma^*\alpha) \left(\tfrac{1 + \mu_\Psi}{2}r_k^2 + F_k -F_*\right).
\end{equation}
It now only remains to take expectation in $x_k$ on both sides of \eqref{E_rkxik_recursion}, and \eqref{T_complexitySC} follows.
\end{proof}

%%%%%%%%%%%%%%%%%%%%%%%%%%%%%%%

\section{High Probability Convergence Result}
\label{S_iterationcomplexity}

Theorem~\ref{T_convergence_rate}
showed that the Algorithm \ref{PCDM} converges to the optimal solution in expectation.
In this section we derive
%sharper
iteration complexity bounds for PCDM %than that presented in   \cite{Richtarik12a},
for obtaining an $\epsilon$-optimal solution with high probability.
Let us mentioned that all
existing \cite{Nesterov12,Richtarik12,
Richtarik12a,Lu13}
high-probability results for serial or parallel CDM require a bounded levelset, i.e. they assume that
\begin{equation}
\label{eg:def:levelset}
\mathcal{L}(x_0)
 =\{ x\in \R^N: F(x) \leq F(x_0) \}
\end{equation}
is bounded.
In Section \ref{sec:HPR:Case1}
we present the first high probability result in the case when the levelset can be unbounded (Corollary \ref{thm:FHPR} and Corollary \ref{thm:RR}).
Then in Section \ref{sec:HPR:Case2}
we derive a sharper high-probability result for PCDM of \cite{Richtarik12a} if a bounded levelset is assumed (i.e.
$\mathcal{L}(x_0)$ is bounded).

\subsection{Case 1: Possibly unbounded levelset}
\label{sec:HPR:Case1}

We begin by presenting Lemma \ref{lemma:1overRho}, which will allow us to state \emph{the first high-probability result} (Corollary \ref{thm:FHPR}) for a PCDM applied to a convex function that \emph{does not require} the assumption of a \emph{bounded} levelset.
\begin{lemma}\label{lemma:1overRho}
Let $x_0$ be fixed and $\{x_k\}_{k=0}^\infty$  be a sequence of random vectors in $\R^N$ such that the conditional distribution of
$x_{k+1}$ on $x_k$ is the same as conditional distribution of $x_{k+1}$ on the whole history $\{x_i\}_{i=0}^\infty$ (hence we have Markov sequence).
Let us define $r_k = \phi_r(x_k)$ and $\xi_k = \phi_\xi(x_k)$
where $\phi_r, \phi_\xi: \R^N \to \R$
are non-negative functions.
Further, let us assume that following two inequalities holds for
any $k$ %(!!!!!! Before r write 1/2!!!!!!!!!!!!!!!!!)
  \begin{align}\label{eq:asfoifojwavfwaefsadfa}
  \Exp\left[ \tfrac12r_{k+1} + \xi_{k+1} | x_k \right] &\leq \tfrac12r_k + (1-\zeta) \xi_k,
  \\
  \Exp[\xi_{k+1}]&\leq \xi_k \label{eq:asjdoiwjfwefa}
\end{align}
with some known $\zeta\in(0,1)$.
Then if
\begin{equation}\label{eq:safvjapowjvgowvgfewa}
  K \geq  \frac1\zeta \left(  \frac{\tfrac12r_0 +\xi_0}{\rho \epsilon }-1\right)
\end{equation}
then
$$
\Prob( \xi_K < \epsilon) \geq 1-\rho.
$$
\end{lemma}
\begin{proof}
Using \eqref{eq:asfoifojwavfwaefsadfa}
we have
$$
\Exp[\xi_k] \leq \Exp\left[\tfrac12r_k +\xi_k\right]
\leq \tfrac12r_0 +\xi_0 - \zeta \sum_{j=0}^{k-1} \Exp[\xi_j]
\overset{\eqref{eq:asjdoiwjfwefa}}{\leq}
\tfrac12r_0 +\xi_0 - k\zeta \Exp[\xi_k].
$$
Hence
\begin{equation}\label{eq:asjvvvvewvfw}
\Exp[\xi_k] \leq \frac{\tfrac12r_0 +\xi_0}{1 + k\zeta}.
\end{equation}
Now, from the Markov inequality we have
\begin{align*}
\Prob( \xi_K \geq \epsilon )  &\leq \frac{\Exp[\xi_K]}{\epsilon}
\overset{
\eqref{eq:asjvvvvewvfw}}{\leq}
\frac{1}{\epsilon} \frac{\tfrac12r_0 +\xi_0}{1
 + K\zeta}
\overset{\eqref{eq:safvjapowjvgowvgfewa}}{\leq} \rho.   \qedhere
\end{align*}
\end{proof}
Naturally, the result $\mathcal{O}(\frac{1}{\epsilon \rho})$ is very pessimistic
and hence one may be concerned about tightness of the lemma.
The following example, indeed, shows that Lemma \ref{lemma:1overRho} is tight, i.e. the bound on $K$ cannot be improved much. (We construct an example that, under the assumptions \eqref{eq:asfoifojwavfwaefsadfa} and \eqref{eq:asjdoiwjfwefa} (i.e., using the analysis of \cite{Lu13}), requires $\mathcal{O}(\frac1{\epsilon\rho})$ iterations.)
\begin{example}[Tightness of Lemma \ref{lemma:1overRho}]\label{Ex_TightUnbounded}
Let us fix some small value of $\rho \in (0,1)$
and assume that $(r_1,\xi_1)$ have following distribution:
\begin{equation*}
 (r_1,\xi_1)
 =
 \begin{cases}
  (0, 0),&\mbox{with probability}\ 1-\rho\\
  (2\vartheta, \epsilon),&\mbox{otherwise},
 \end{cases}
\end{equation*}
where $\vartheta$ is chosen in such a way that \eqref{eq:asfoifojwavfwaefsadfa}
is satisfied. Then, we can chose it as follows
\begin{equation*}
   \rho   (\vartheta +\epsilon ) = \frac12 r_0+(1-\zeta) \xi_0
   \quad
   \Rightarrow
   \quad
   \vartheta = \frac{\frac12 r_0+(1-\zeta) \xi_0}{\rho} - \epsilon.
\end{equation*}
Now we define, for $k=1,2,3, \dots$
\begin{equation*}
 (r_{k+1},\xi_{k+1})
 =
 \begin{cases}
  (r_k - 2 \zeta \epsilon  , \epsilon ),&\mbox{if}\ r_k \geq  2 \zeta \epsilon\\
  (0, 0),&\mbox{otherwise}.
 \end{cases}
 \end{equation*}
 Now it is easy to verify that for
 $$
 K \eqdef \left\lfloor\frac{  \vartheta}{\zeta \epsilon}
 \right\rfloor
  =\frac1{\zeta} \left\lfloor    \frac{\tfrac12 r_0+(1-\zeta) \xi_0}{\rho  \epsilon }
  -1 \right\rfloor
 $$
 we have that $\Prob( \xi_{K} \geq \epsilon) \geq \rho$.
\end{example}

\begin{corollary}[High probability result without bounded levelset]
\label{thm:FHPR}
If we use Lemma \ref{lemma:1overRho} with  $\tfrac12r_k =\phi_r(x_k) = \frac12 \|x_k-x_*\|^2_v$, $\xi_k=\phi_\xi(x_k)=F(x_k)-F_*$  and  $\zeta=\alpha=\frac{\Exp|\hatS|}{n}$ then
we obtain that
$$\forall K \geq \frac{n}{\Exp|\hatS|} \left(  \frac{\frac12 \|x_0-x_*\|_v^2 +F(x_0)-F_*}{\rho \epsilon }-1\right)$$ we have
$\Prob(F(x_K)-F_* <\epsilon) \geq 1-\rho$.
\end{corollary}

The negative aspect of Corollary \ref{thm:FHPR} is the fact that one needs
$\mathcal{O}(\frac1{\rho})$ iterations, whereas classical results under the bounded levelset assumption require only $\mathcal{O}(\log\frac1\rho)$ iterations.

\paragraph{Multiple run strategy.}
Now we present a restarting strategy
\cite{Richtarik12} trick which will give us high probability result
$\mathcal{O}(\log\frac1\rho)$.
\begin{lemma}
Let
$\{x_k\}_{k=0}^\infty$,
$\{r_k\}_{k=0}^\infty$ and $\{\xi_k\}_{k=0}^\infty$
be the same as in Lemma
\ref{lemma:1overRho}.
Assume that we
observe $r = \lceil \log \frac 1{\rho}
\rceil$ different random and independent realisations
of this sequence always starting from $x_0$, i.e. for any $k$ we have
observed $x_k^1, x_k^2, \dots, x_k^r$.
Then if
$$
 K \geq
 \frac1\zeta \left(   \frac{\frac12 r_0 +\xi_0}{ \epsilon (1/e) }-1\right)
$$
then
$$
\Prob\left(
   \min_{l \in \{1,2,\dots, r\}}
   \xi_K^l < \epsilon\right)
   \geq 1-\rho.
$$
\end{lemma}
\begin{proof}
Because the realisation are independent
then for
any $l \in \{1,2,\dots,r\}$
we have
from Lemma~\ref{lemma:1overRho}
that
$\Prob( \xi_K^l \geq \epsilon) \leq \frac 1e$.
Hence
\begin{align*}
\Prob\left(
   \min_{l \in \{1,2,\dots, r\}}
   \xi_K^l \geq \epsilon\right)
 &=
\Prob\left(
   \xi_K^1 \geq \epsilon
   ,
   \xi_K^2 \geq \epsilon,
   \dots,
   \xi_K^r \geq \epsilon
   \right)
 =
    \prod_{l \in \{1,2,\dots, r\}}
\Prob\left(
   \xi_K^l \geq \epsilon\right)
   \leq
   \left(\frac1e\right)^r
   \leq \rho. \qedhere
\end{align*}
\end{proof}

\begin{corollary}
\label{thm:RR}
If we run PCDM $r=\lceil \log \frac 1{\rho}
\rceil$ many times for $K\geq \frac{n}{\Exp[|\hatS|]} \left(   \frac{\frac12\|x_0-x_*\|_v^2 +F(x_0)-F_*}{ \epsilon (1/e) }-1\right)$
each, then
the best solution we get, indexed $l\in \{1,2,\dots,r\}$,  satisfies
$\Prob( F(x_K^l)-F_* < \epsilon)\geq 1-\rho$.
Hence, in total we need
$\left\lceil
\frac{n}{\Exp[|\hatS|]} \left(   \frac{\frac12\|x_0-x_*\|_v^2 +F(x_0)-F_*}{ \epsilon (1/e) }-1\right)
\right\rceil
\lceil \log \frac1\rho \rceil
\sim \mathcal{O}\left(\log \frac1\rho\right)
$
iterations of PCDM.
\end{corollary}

\subsection{Case 2: Bounded levelset}
\label{sec:HPR:Case2}

The next result, Theorem \ref{T_Complexity}, obtains the rate $\mathcal{O}(\log\frac1\rho)$, under the assumption that the levelset is bounded.
However, some results will hold only for a modified version of Algorithm \ref{PCDM}. In particular, we now present Algorithm \ref{PCDM-Monotinic}.
\begin{algorithm}[H]
\caption{PCDM-M: Parallel Coordinate Descent Method \cite{Richtarik12a} }\label{PCDM-Monotinic}
\begin{algorithmic}[1]

\State choose initial point $x_0 \in\R^N$
\For{$k = 0,1,2,\dots$}
\State randomly choose set of blocks $S_k \subseteq \{1,\dots,n\}$
\label{alg:line:sampling}
  \For{$i \in S_k$ ({\bf in parallel})}
  \State  compute
  %\begin{equation}
%\label{E_argminhi}
$\vc{h(x_k)}{i} = \arg \min_{t \in \R^{N_i}} \Big\{\langle \vc{(\nabla f(x_k))}{i},t \rangle + \frac{  v_i}{2} \|t\|_{(i)}^2 + \Psi_i(\vc{x_k}{i}+t)\Big\}
$%\end{equation}
\label{alg:line:stepsize}
  \EndFor

\If{$F(x_k +  \sum_{i\in S_k} U_i \vc{h(x_k)}{i} ) \leq F(x_0)$}
\State apply the update: $x_{k+1} \gets x_k + \sum_{i\in S_k} U_i \vc{h(x_k)}{i}$
\Else
  \State set  $x_{k+1} \gets x_k$
\EndIf

\EndFor

 \end{algorithmic}
\end{algorithm}
Notice that the first 6 steps of Algorithm~\ref{PCDM-Monotinic} are exactly the same as those of Algorithm~\ref{PCDM}. However, Algorithm \ref{PCDM-Monotinic} forces the iterates to stay in $\mathcal{L}(x_0)$ (steps 7--11).

\paragraph{Distance to the optimal solution set.}
%The following quantity measures the distance between $x_0$ and the set of optimal solutions $X^*$, of problem \eqref{D_F} that will appear in our complexity results:
%\begin{equation}
%\label{D_R}
%R_{v,0} \eqdef \min_{x_* \in X^*} \|x_0 %- x_*\|_v.
%\end{equation}
%Some results in literature use a different measure, which gives the distance to the levelset of $F$, and is defined as follows
In some of the results derived in this Section we need the distance to the optimal solution set, inside the levelset, to be finite, i.e.
\begin{equation}
\label{D_barR}
\mathcal{R}_{v,0}\eqdef \max_{x\in\mathcal{L}(x_0)} \Big\{\max_{x_* \in X^*}\|x-x_*\|_v\Big\} < \infty.
\end{equation}
%where $\mathcal{L}(x_0)=\{x\in\R^N: F(x) \leq F(x_0)\}$.
Note that for any $x_* \in X^*$ (where $X^*$ is a set of optimal solutions) it
trivially holds that $\|x_0 - x_*\|_v \leq \mathcal{R}_{v,0}$. Moreover, for some problems the levelset can be unbounded, in which case $\mathcal{R}_{v,0}$ is infinite, whereas if $X^* \neq \emptyset$ then $\|x_0 - x_*\|$ is \emph{always finite}.

\begin{theorem}
\label{T_Complexity}
Let $\{x_k\}_{k \geq 0}$ be a sequence of iterates generated by
\begin{itemize}
 \item PCDM (Algorithm \ref{PCDM}), if $F$ is strongly convex with $\mu_f(w) + \mu_\Psi(w) >0$ or $F$ is convex and a monotonic ESO is used,
 \item PCDM-M (Algorithm \ref{PCDM-Monotinic}), if $F$ is convex and a non-monotonic ESO is used.
\end{itemize}
Let $0 < \epsilon < F(x_0) - F_*$ and $\rho \in (0,1)$ be chosen arbitrarily. Define $\alpha = \frac{\Exp[|\hatS|]}{n}$, and let
\begin{eqnarray}
\label{D_c}
c &\eqdef& \max\{  \mathcal{R}_{v,0}^2,F(x_0)-F_*\}.
\end{eqnarray}
 Then
\begin{itemize}
\item[(i)] if $F$ is convex and we choose
\begin{equation}
\label{D_Kconvex}
K \geq \frac{2 c}{\alpha \epsilon}\left(1 + \log\left(\frac{\frac{1}{2} \|x_0-x_*\|_v^2 + F(x_0) - F_*}{2c\rho} \right) \right) + 2 - \frac{1}{\alpha},
\end{equation}
\item[(ii)] or if $F$ is strongly convex with $\mu_f  + \mu_\Psi  >0$ and we choose
\begin{equation}
\label{D_Kstronglyconvex}
K \geq \frac{1 + \mu_f  + 2\mu_\Psi }{2\alpha(\mu_f +\mu_\Psi )}\log\left(\frac{\frac{1+\mu_\Psi }{2}\|x_0-x_*\|_v^2+F(x_0)-F_*}{\epsilon \rho}\right)
\end{equation}
\end{itemize}
then
\begin{equation}
\label{E_highprob}
\mathbf{P}(F(x_K) - F_* < \epsilon) \geq 1-\rho.
\end{equation}
\end{theorem}
\begin{proof}
The proof proceeds as in \cite[Theorem~1]{Richtarik12}. For convenience, let $\xi_k\eqdef F(x_k)-F_*$ and define
\begin{equation*}
\xi_k^\epsilon =
\begin{cases}
\xi_k, & \text{if } \xi_k \geq \epsilon,\\
0, & \text{otherwise.}
\end{cases}
\end{equation*}
Notice that $\xi_k^\epsilon < \epsilon \Leftrightarrow \xi_k < \epsilon, k \geq 0$. Using the Markov inequality,
\begin{equation}
\label{E_Markov}
\Prob(F(x_k) - F_* \geq \epsilon) = \Prob(\xi_k \geq \epsilon) = \Prob(\xi_k^\epsilon \geq  \epsilon) \leq \tfrac{1}{\epsilon}\Exp[\xi_{k}^\epsilon],
\end{equation}
so it suffices to find $K$ such that
\begin{equation}
\label{D_highprob}
  \Exp[\xi_{K}^\epsilon] \leq \epsilon \rho.
\end{equation}
Using an ESO and Lemma 17 in \cite{Richtarik12a}
will give us
\begin{equation}
\label{eq:asdfjoiwajfcaw}
\Exp[\xi_{k+1} | x_k]
\leq \left(1-
\frac{\alpha \xi_k}{2c} \right) \xi_k.
\end{equation}
It is easy to verify that
\eqref{eq:asdfjoiwajfcaw}
and the definition of $\xi_k^\epsilon$
lead to (see the proof of \cite[Theorem~1]{Richtarik12})
\begin{equation*}
\Exp[\xi_{k+1}^\epsilon|x_k] \leq \left(1 - \frac{\alpha \epsilon }{2 c} \right)\xi_{k}^\epsilon, \qquad \forall k \geq 0.
\end{equation*}
Taking expectation with respect to $x_{k}$ on both sides of the above we get
\begin{equation}
\label{E_exp}
\Exp[\xi_{k+1}^\epsilon] \leq \left(1 - \frac{\alpha \epsilon }{2 c} \right)\Exp[\xi_{k}^\epsilon], \qquad \forall k \geq 0.
\end{equation}
In addition, using \eqref{T_complexityC} and the relation $\xi_{k}^\epsilon\leq \xi_{k}$, we have
\begin{equation}
\label{E_xike}
\Exp[\xi_{k}^\epsilon] \leq\frac{1}{1+\alpha k} \left(\frac{1}{2}\|x_0-x_*\|_v^2 + \xi_0 \right), \qquad \forall k \geq 0.
\end{equation}
Now for any $t>0$, let
\begin{equation}
K_1 = \left\lceil \frac{1}{\alpha}\left(\frac{1}{ t \epsilon} \left( \frac{1}{2} \|x_0-x_*\|_v^2 + \xi_k \right) -1\right) \right\rceil, \qquad K_2 = \left\lceil \frac{2 c}{\alpha\epsilon}\log\left(\frac{t}{\rho}\right)\right\rceil.
\end{equation}
It follows from \eqref{E_xike} that $\Exp[\xi_{K_1}^\epsilon] \leq t \epsilon$, which together with \eqref{E_exp} implies that
\begin{equation}
\Exp[\xi_{K_1+K_2}^\epsilon|x_{K_1}] \leq \left(1 - \frac{\alpha \epsilon}{2 c}\right)^{K_2} \Exp[\xi_{K_1}^\epsilon ] \leq \left(1 - \frac{\alpha \epsilon}{2 c}\right)^{K_2}t\epsilon \leq \rho\epsilon.
\end{equation}
Notice that, by \eqref{E_exp}, the sequence $\Exp[\xi_{k}^\epsilon]$ is decreasing. Hence, we have
\begin{equation}
\Exp[\xi_{k}^\epsilon ] \leq \rho \epsilon, \qquad \forall k \geq K(t),
\end{equation}
where
\begin{equation}
K(t) \eqdef \frac{1}{\alpha}\left(\frac{1}{t \epsilon}\left(\frac{1}{2}\|x_0-x_*\|_v^2 + F(x_k) - F_* \right) -1\right) + \frac{2 c}{\alpha\epsilon}\log\left(\frac{t}{\rho}\right) + 2.
\end{equation}
It is easy to verify that
\begin{equation}
t_* (= \arg \min_{t>0} K(t) ) \eqdef \frac{1}{2c} \left(\frac{1}{2}\|x_0-x_*\|_v^2 + F(x_{0}) - F_* \right),
\end{equation}
Because $K \geq K(t_*)$, we see that \eqref{D_highprob} holds and the proof of (i) is complete.

Now we prove (ii). For convenience, set $\mu_\Psi \equiv \mu_\Psi(w)$. Then from \eqref{T_complexitySC}, we have
\begin{equation}
 \Exp[\xi_{k+1}|x_k] \leq \left(1 - \alpha\gamma \right) \left(\frac{1+\mu_\Psi}{2}\|x_0-x_*\|_v^2 + \xi_k \right),
\end{equation}
where $0 <\gamma \leq 1$ is defined in \eqref{D_gamma}. Taking expectation in $x_k$ (and using recursion) gives $\Exp[\xi_{k+1}] \leq \left(1 - \alpha\gamma \right)^k \left(\frac{1+\mu_\Psi}{2}\|x_0-x_*\|_v^2 + \xi_0 \right)$. Finally, using the Markov inequality \eqref{E_Markov}, and $K$ given in \eqref{D_Kstronglyconvex}, we have
\begin{equation}
  \Prob(\xi_K \geq \epsilon) \leq \frac{1}{\epsilon}\Exp[\xi_K]\leq \frac{1}{\epsilon}(1 - \alpha\gamma)^K \left(\frac{1+\mu_\Psi}{2}\|x_0-x_*\|_v^2 + \xi_0 \right)\leq \rho,
\end{equation}
and the result follows.
\end{proof}

In this Section we have presented three new convergence results for PCDM. The first result shows that, using the analysis in \cite{Lu13}, PCDM obtains a $O(\frac1\rho)$ rate when the levelset is unbounded for a single run strategy. The second result shows that PCDM obtains a $O(\log \frac1\rho)$ rate for a restarting strategy.

On the other hand, if the levelset is bounded, we have shown that PCDM achieves a rate of $O(\log \frac1\rho)$. It is still an open problem to determine whether PCDM can achieve a rate of $O(\log \frac1\rho)$ for a single run strategy when the levelset is unbounded.

\section{Discussion}
\label{sec:discusion}

\subsection{Comparison of the convergence rate results}
\label{S_comparisonrate}
We have the following remarks on comparing the results in Theorem \ref{T_convergence_rate} with those in \cite{Richtarik12a}.

\subsubsection{Comparison in the convex case}

For problem \eqref{D_F}, an expected-value type of convergence rate is not presented explicitly in \cite{Richtarik12}, although it can be derived from the following relation (that is stated in \cite{Richtarik12a} and proved in \cite[Theorem~1]{Richtarik12}):
\begin{equation}\label{exp1}
\E[F(x_{k+1})-F^*|x_k] \leq (F(x_k)-F^*) -  \alpha\frac{(F(x_k)-F^*)^2}{2c}, \qquad \forall k \geq 0,
\end{equation}
where $c$ is defined in \eqref{D_c}. Taking expectation on both sides of \eqref{exp1} and using a similar argument as that in \cite{Nesterov12}, gives
\begin{equation}\label{exp2}
\E[F(x_k)-F^*|x_{k-1}] \leq \frac{2c(F(x_0)-F^*)}{2c + \alpha k (F(x_0)-F^*)}, \qquad \forall k \geq 0.
\end{equation}
Let $a$ and $b$ denote the right hand side of \eqref{T_complexityC} and \eqref{exp2} respectively. By the definition of $c$ and the relation $\|x_0 - x_*\|_v \leq \mathcal{R}_{v,0}$, we see that when $k$ is sufficiently large,
\begin{equation}
\frac{b}{a} \approx \frac{4c}{ \|x_0 - x_*\|_v^2 + 2(F(x_0) - F^*)} \geq \frac{4}{3}.
\end{equation}

\subsubsection{Comparison in the strongly convex case}

For the special case of \eqref{D_F} where at least one of $f$ and $\Psi$ is strongly convex (i.e., $\mu_f+\mu_\Psi > 0$), Richtarik and Takac \cite{Richtarik12a} showed that for all $k \geq 0$, there holds
\begin{equation}
\E[F(x_k)-F^*|x_{k-1}] \leq \left( 1 - \alpha \frac{\mu_f + \mu_\Psi}{1 + \mu_\Psi} \right)(F(x_0)-F^*).
\end{equation}
It is not hard to observe that
\begin{equation}
\frac{2(\mu_f+\mu_\Psi)}{1 + \mu_f + 2\mu_\Psi} > \frac{\mu_f + \mu_\Psi}{1 + \mu_\Psi}.
\end{equation}
% Further,
% \begin{equation*}
% \frac{\beta + \mu_\Psi}{2}R_0^2 = \frac{\beta + \mu_\Psi}{2}\|x_0 - x_*\|_w^2  \overset{\eqref{strongly_convex_4}}{\leq} \frac{2}{\mu_f+\mu_\Psi} (F(x_0)-F^*)
% \end{equation*}
Recall that $\gamma$ is defined in \eqref{D_gamma}. Then it follows that for sufficiently large $k$ one has
\begin{eqnarray*}
\left(1 - \alpha \gamma \right)^k\left(\frac{1+\mu_\Psi}{2}R_0^2 + F(x_0)-F^*\right)
&\overset{\eqref{strongly_convex_4}}{\leq}& \left(1 - \alpha \gamma \right)^k\left(\frac{1+\mu_f+\mu_\Psi}{\mu_f+\mu_\Psi}\right) (F(x_0)-F^*).
\end{eqnarray*}

\subsection{Comparison of the iteration complexity results}\label{S_Discussion_complexity}

Here we compare the results in Theorem \ref{T_Complexity} with those in  \cite{Richtarik12a}.

\paragraph{Comparison in the convex case.}

For any $0<\epsilon<F(x_0)-F_*$ and $\rho \in (0,1)$, Richt\'arik and Tak\'a\v{c} \cite{Richtarik12a} showed that \eqref{E_highprob} holds for all $k \geq \tilde{K}$ where
\begin{equation}
  \tilde{K}\eqdef \frac{2c}{\alpha \epsilon}\left( 1+\log\left(\frac{1}{\rho}\right)\right) + 2- \frac{2c}{\alpha(F(x_0) -F_*)}.
\end{equation}
Using the definition of $c$ and the fact that $\|x_0 - x_*\|_v \leq\mathcal{R}_{v,0}$ we observe that
\begin{equation}
  \tau \eqdef \frac{\|x_0 - x_*\|_v^2 + 2\xi_0}{4 c}\leq \frac{3}{4}.
\end{equation}
By the definitions of $K$ and $\tilde{K}$ we have that for sufficiently small $\epsilon >0$,
\begin{equation}
  K - \tilde{K}\approx \frac{2 c \log \tau}{\alpha \epsilon} \leq -\frac{2 c \log(4/3)}{\alpha \epsilon}.
\end{equation}
In addition,  $\|x_0 - x_*\|_v$ can be much smaller than $\mathcal{R}_{v,0}$ and thus $\tau$ can be very small. It follows from the above that $K$ can be significantly smaller than $\tilde{K}$.

\paragraph{Comparison in the strongly convex case.}

In the strongly convex case (i.e., $\mu_f(w)+\mu_\Psi(w) >0$), Richt\'arik and Tak\'a\v c showed that \eqref{E_highprob} holds for all $k \geq \hat{K} $ where
\begin{equation*}
  \hat{K} \eqdef \frac{1}{\alpha}\frac{1 + \mu_\Psi(w)}{\mu_f(w) + \mu_\Psi(w)}\log \left(\frac{F(x_0)-F_*}{\epsilon \rho} \right).
\end{equation*}
We can see that for $\rho$ or $\epsilon$ sufficiently small we have
\begin{equation}
  \frac{K}{\hat{K}} \leq \frac{1 + \mu_f(w) + \mu_\Psi(w) }{2(1+\mu_\Psi(w))}\leq 1,
\end{equation}
because $\mu_f \leq 1$, which demonstrates that $K$ is smaller than $\hat{K}$.

\section{Numerical experiments}\label{sec:numerical}

In this Section we present preliminary computational results. The purpose of these experiments is to provide a numerical comparison of the performance of PCDM, under the different ESOs summarized in Appendix \ref{sec:asdfoijwofwaefwa}.

\paragraph{Least squares.}
Consider the following convex optimization problem
$ \displaystyle
\min_{x \in \R^N} \frac12 \|Ax-b\|_2^2,
$
where $A\in \R^{8\cdot 10^3 \times 2 \cdot 10^3}$.
Each row has between $1$ and $\omega=20$ nonzero elements (uniformly at random). For simplicity, we normalize (in $\ell_2$ norm) all the columns of $A$.
The value of $\sigma=\lambda_{\max}(A^TA) = 10.48$.
We have compared 5 different approaches which are given in Table~\ref{tab:leastSquareESOS}.
\begin{table}[htp]
 \caption{Approaches used in the numerical experiments.}
 \label{tab:leastSquareESOS}
 \centering
 \begin{tabular}{l|c|p{8cm}}
 \multicolumn{1}{c|}{Name} &   $v $ & \multicolumn{1}{|c}{Note} \\ \hline\hline
 BKBG &  $v_{BKBG}=L$ & This is na\"ive approach, which was proposed in
 \cite{Bradley} and \cite{richtarik2012efficient}.
Note that this is not ESO.
 \\
 \hline
 RT-P
 &$v_{RT-P}=(1+\frac{(\omega-1)(\tau-1)}{\max\{1,n-1\}}) L$
 & Theorem \ref{thm:niceESO}, originally derived in \cite{Richtarik12a}.
 \\\hline
 RT-D
 & $v_{RT-D}=(1+\frac{(\sigma-1)(\tau-1)}{\max\{1,n-1\}})L$
 & Derived in \cite{richtarik2013distributed} as a special case for $C=1$.
 \\\hline
 FR &  $v_{FR}=\hat L$& Theorem \ref{thm:niceNewESO}, proposed in \cite{Fercoq:accelerated} and {\bf generalized in this paper} (Theorem \ref{thm:NewESOforDUS}).
 \\\hline
 NC & $v_{NC} = \tilde L$& Theorem \ref{thm:IonESO}, proposed in \cite{necoara2013distributed}.\\
 \hline
\end{tabular}
\end{table}
Parameter $\tau=512$ and hence   $1+\frac{(\omega-1)(\tau-1)}{\max\{1,n-1\}}=5.856$ for RT-P and $1+\frac{(\sigma-1)(\tau-1)}{\max\{1,n-1\}}=3.424$ for RT-D approach.
The distribution of vectors $v$ can be found in Figure \ref{fig:LS} (right). Figure \ref{fig:LS} shows the evolution of $F(x_k)-F^*$
for all 5 methods. Note that the BKBK did not converge.
The speed of RT-P, RT-D and FR is quite similar and NC is approximately 3 times worse because $v_{NC} \approx 3.22 v_{FR}$.
%\begin{figure}[htp]
% \centering
% \includegraphics[width=3in]{matlab/LS_evolution.eps}
% \includegraphics[width=3in]{matlab/LS_dis.eps}
% \caption{Evolution of $F(x_k)-F^*$ for 5 different methods (left) and distribution of $v$ (right).}
% \label{fig:LS}
%\end{figure}
\begin{figure}[htp]
 \centering
 \includegraphics[width=3in]{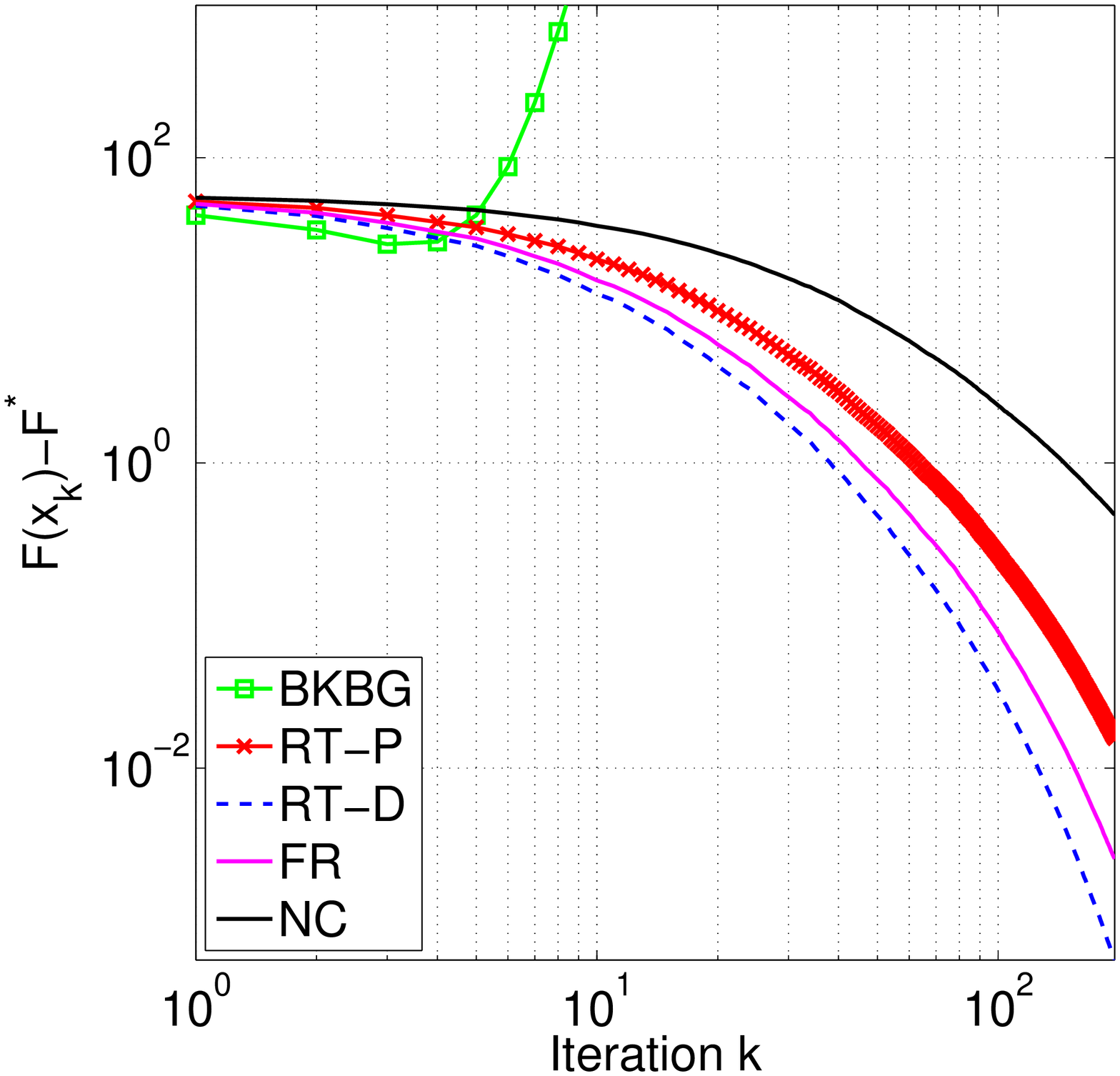}
 \includegraphics[width=3in]{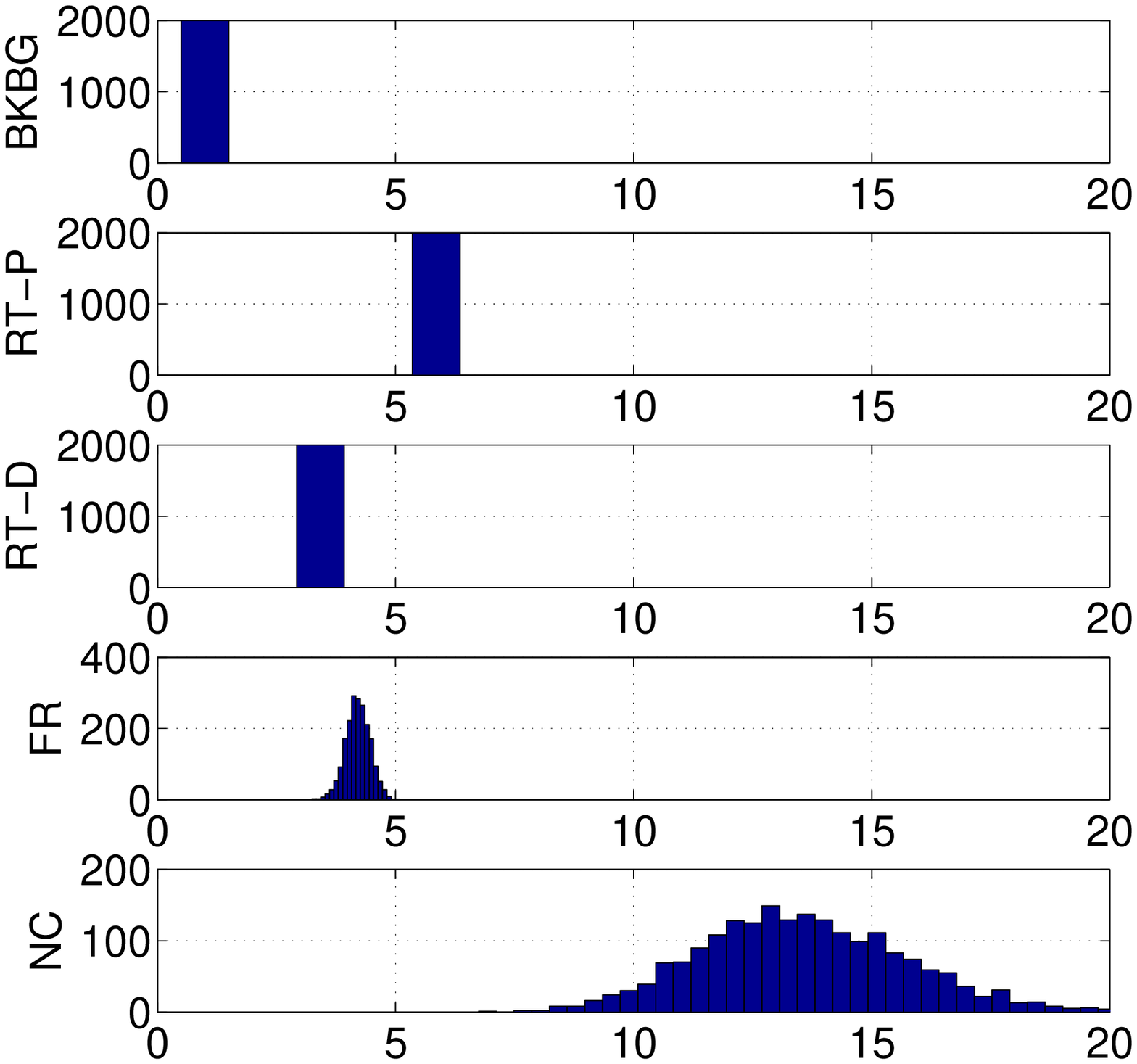}
 \caption{Evolution of $F(x_k)-F^*$ for 5 different methods (left) and distribution of $v$ (right).}
 \label{fig:LS}
\end{figure}

\paragraph{SVM dual.}
%\begin{figure}[htp]
% \centering
% \begin{tabular}{c||cc}
%\smash{\raisebox{100pt}{\rot{$\tau=32$}}}    &
% \includegraphics[width=3in]{matlab/SVM_32_evolution.eps} &
% \includegraphics[width=3in]{matlab/SVM_32dis.eps}
% \\ \hline
%\smash{\raisebox{100pt}{\rot{$\tau=256$}}}  &
% \includegraphics[width=3in]{matlab/SVM_256_evolution.eps} &
% \includegraphics[width=3in]{matlab/SVM_256dis.eps}
%\end{tabular}
% \caption{Comparison of evolution of $G(x_k)$ for various methods and the distribution of $v$.}
% \label{fig:SVM}
%\end{figure}
\begin{figure}[htp]
 \centering
 \begin{tabular}{c||cc}
\smash{\raisebox{100pt}{\rot{$\tau=32$}}}    &
 \includegraphics[width=3in]{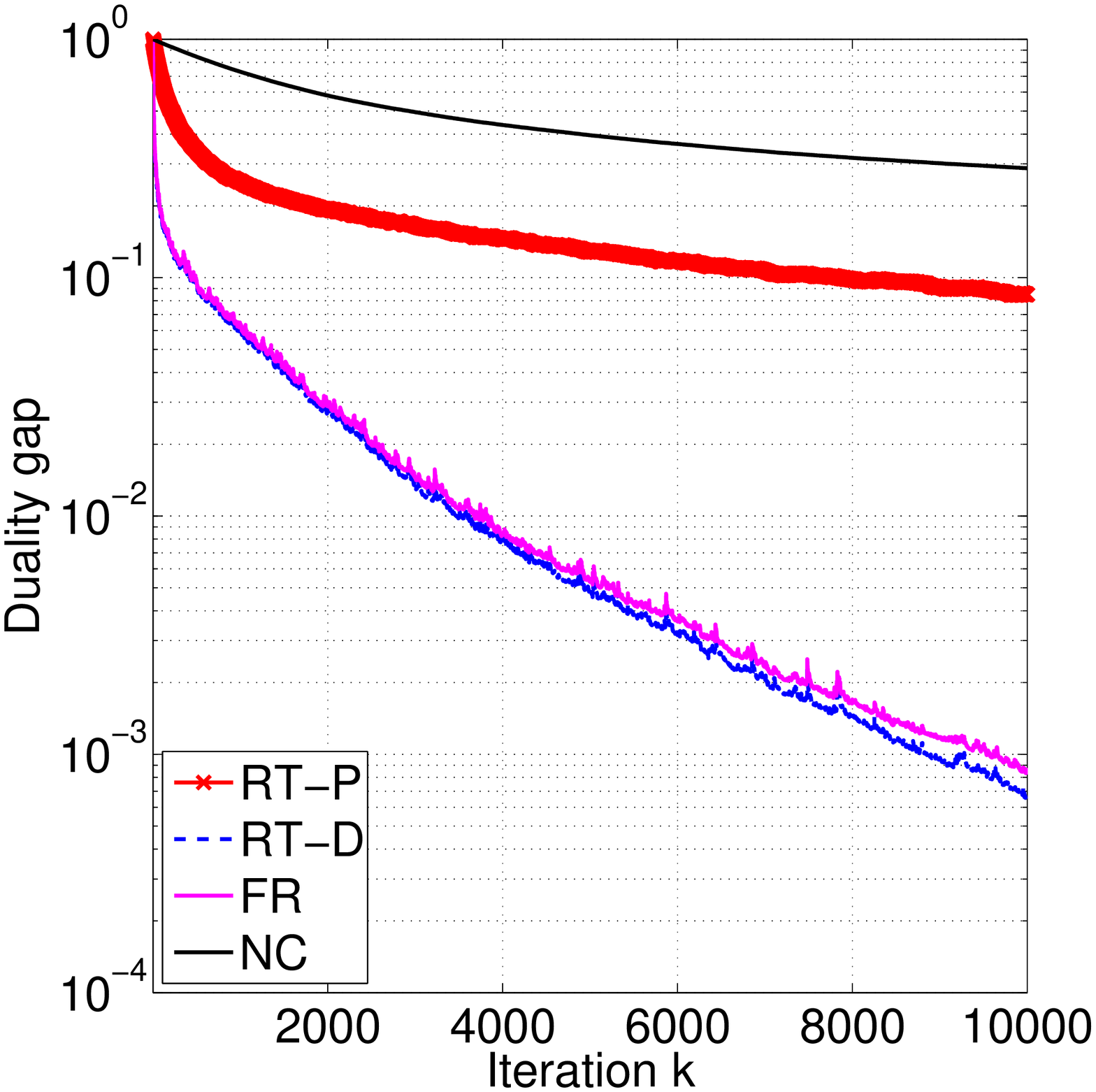} &
 \includegraphics[width=3in]{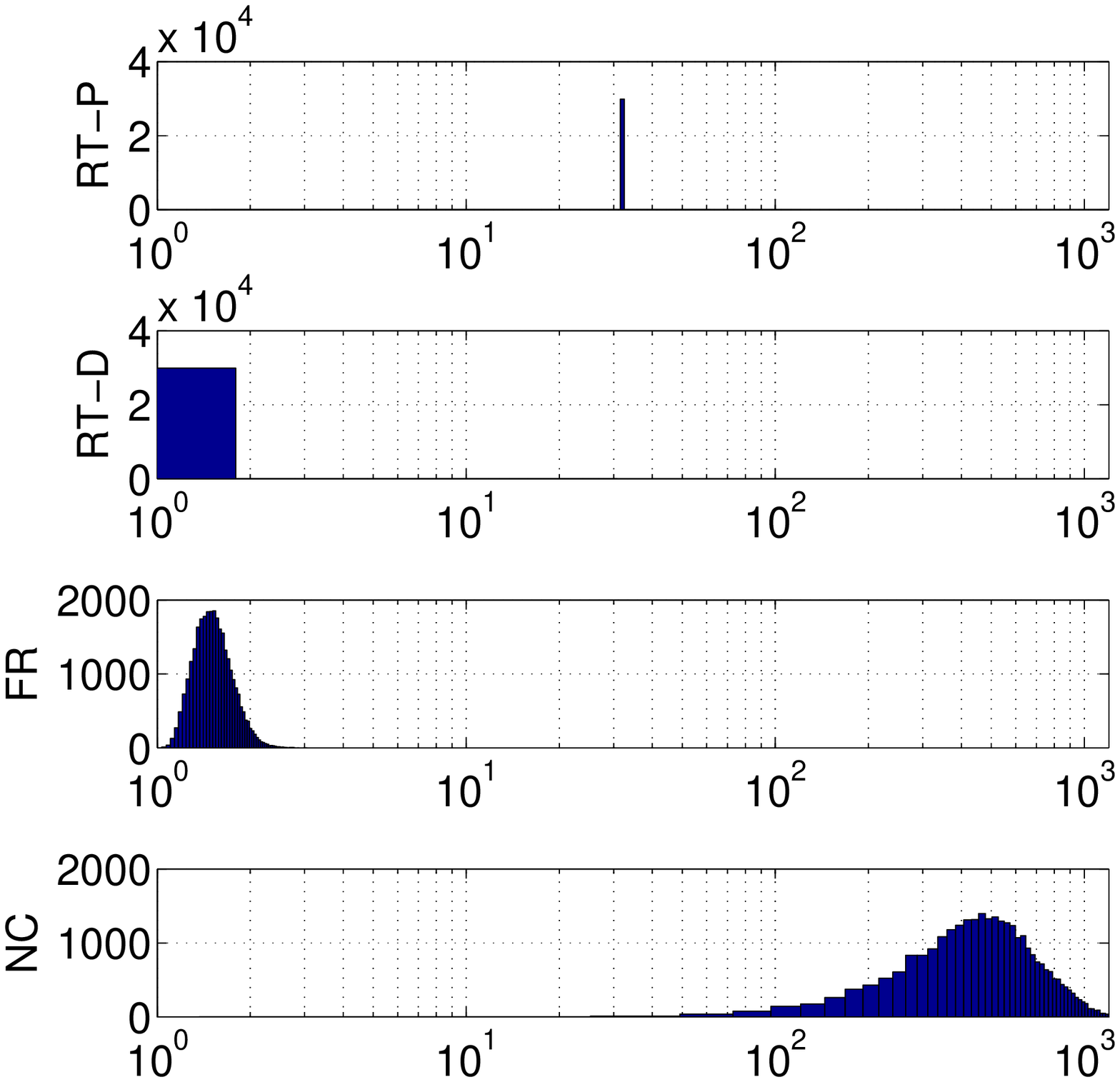}
 \\ \hline
\smash{\raisebox{100pt}{\rot{$\tau=256$}}}  &
 \includegraphics[width=3in]{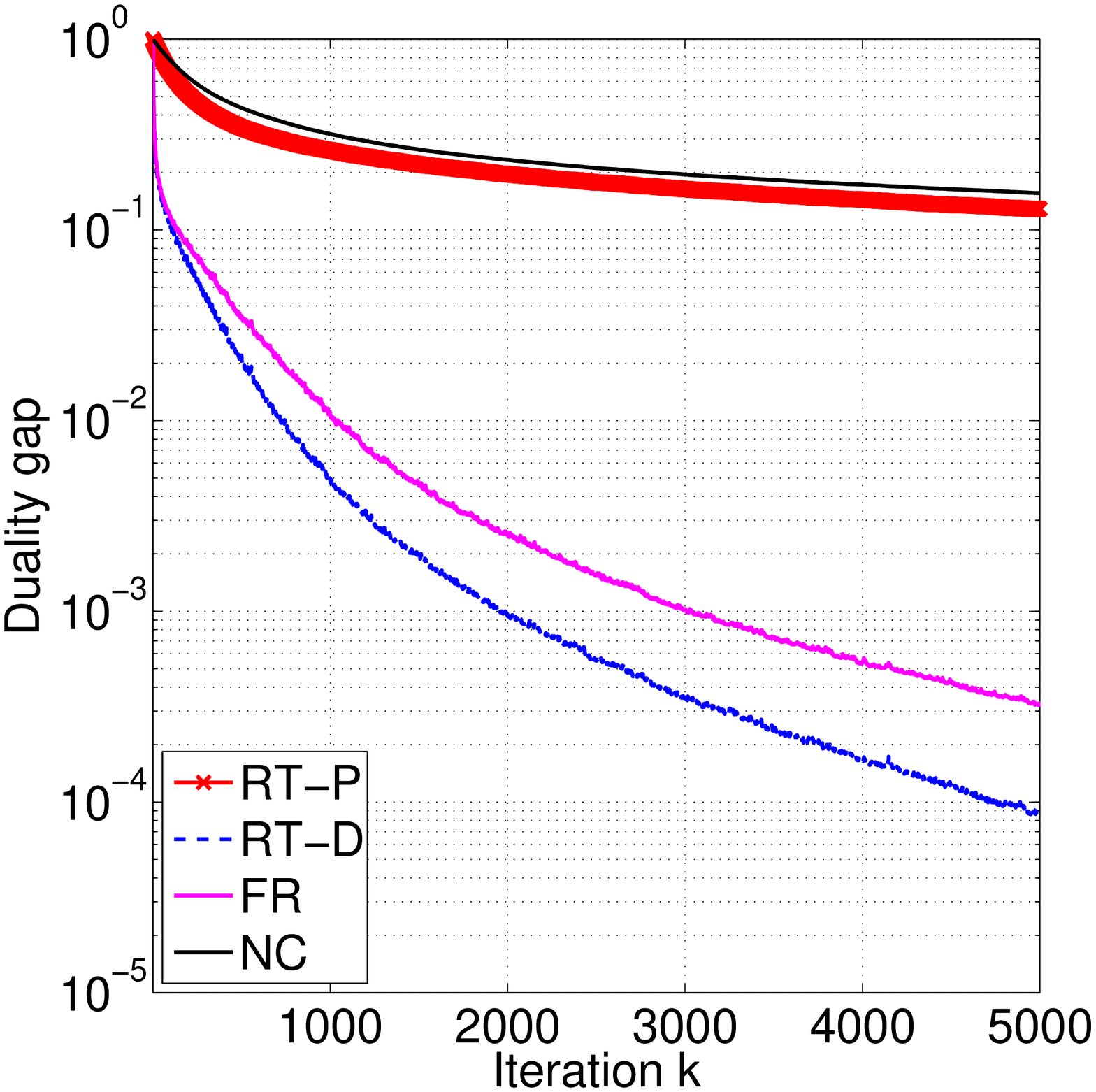} &
 \includegraphics[width=3in]{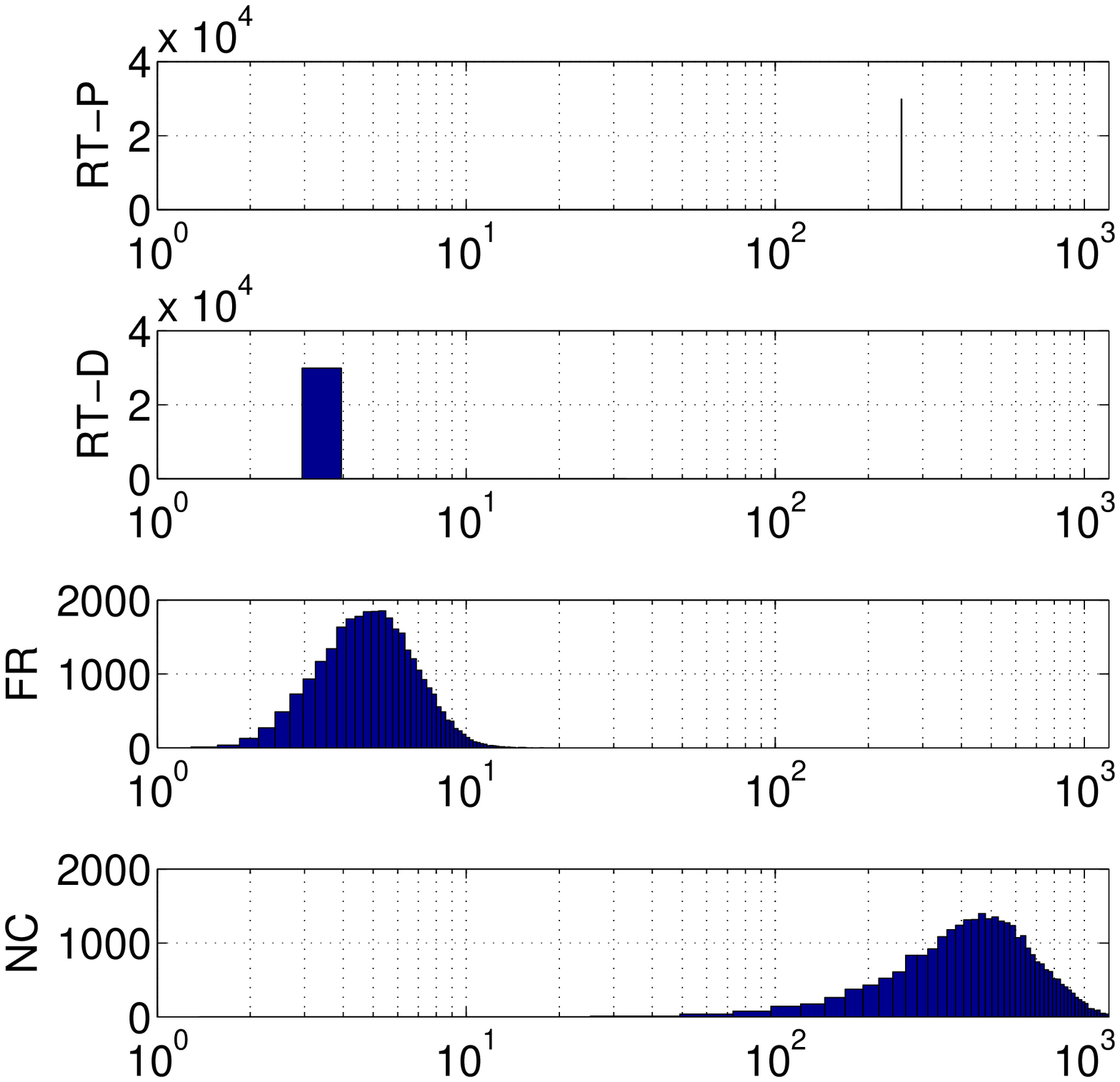}
\end{tabular}
 \caption{Comparison of evolution of $G(x_k)$ for various methods and the distribution of $v$.}
 \label{fig:SVM}
\end{figure}
In this experiment we compare 4 methods from Table \ref{tab:leastSquareESOS} (we have excluded the na\"ive approach because it usually diverges for large $\tau$)
on a real-world dataset \emph{astro-ph}, which consists of data from papers in physics \cite{pegasos}.
This dataset has $29,882$ training samples and a total of $99,757$ features.
This dataset is very sparse. Indeed, each sample uses on average only $77.317$ features and each sample belongs to one of two classes. Hence,
one might be interested in finding a hyperplane that separates the samples into their corresponding classes.
The optimization problem can be formulated as follows:
\begin{equation}\label{eq:SVM_P}
\min_{w} P(w)\eqdef \frac{\lambda}{2} \|w\|_2^2 + \frac1n \sum_{i=1}^N \max\{0, 1-\vc{y}{i} a_{(i)}^T w\},
\end{equation}
where $\vc{y}{i} \in \{-1,1\}$ is the label of the class to which sample $a_{(i)}\in\R^m$ belongs.

While problem formulation \eqref{eq:SVM_P} does not fit our framework (the nonsmooth part is nonseparable) the dual formulation (see \cite{hsieh2008dual,ShalevShawartzZhang,takac2013mini}) does:
%The nonsmooth part of this optimization problem is not separable, but the dual formulation fits our framework (a smooth objective plus an indicator function of box, which is separable). It is easy to verify that the dual objective has following form (see \cite{hsieh2008dual,ShalevShawartzZhang,takac2013mini})
\begin{equation}\label{eq:SVM_D}
 \max_{x  \in[0,1]^N}  D(x)\eqdef \frac1N {\bf 1}^T x - \frac{1}{2\lambda N^2} x^T Q x,
\end{equation}
where $Q\in \R^{N\times N}, Q_{i,j} = \vc{y}{i} \vc{y}{j} \langle a_{(i)}, a_{(j)}\rangle$. In particular, problem formulation \eqref{eq:SVM_D} is the sum of a smooth term, and the restriction $x  \in[0,1]^N$ can be formulated as a (block separable) indicator function.
In this dataset, each sample is normalized, hence $L=(1,\dots,1)^T$.

For any dual feasible point $x$ we can obtain a primal feasible point $w(x) = \frac1{\lambda n} \sum_{i=1}^N \vc{x}{i} \vc{y}{i} a_{(i)}$. Moreover, from strong duality we know that if $x^*$ is an optimal solution of \eqref{eq:SVM_D}, then $w^* = w(x^*)$ is optimal for problem \eqref{eq:SVM_P}.
Therefore, we can associate a gap $G(x)=P(w(x))-D(x)$ to each feasible point $x$, which measures the distance of the objective value from optimality. Clearly $G(x^*)=0$.

Figure \ref{fig:SVM} (left) shows the evolution of $G(x_k)$ as the iterates progress, and the distribution of ESO parameter $v$ for different choice of $\tau\in \{32, 256\}$.
Naturally, as $\tau$ increases, the distribution of $v, \hat v$ shifts to the right, whereas the distribution of $\tilde v$ is not influenced by changing $\tau$.
The value of important parameters for other methods are
$\sigma=287.273$ and $\omega=29881$.
For $\tau=32$ we have $1+\frac{(\omega-1)(\tau-1)}{\max\{1,n-1\}}=31.998$ for RT-P and $1+\frac{(\sigma-1)(\tau-1)}{\max\{1,n-1\}}=1.296$ for RT-D approach
and
for
$\tau=256$ we have $1+\frac{(\omega-1)(\tau-1)}{\max\{1,n-1\}}= 255.991$ for RT-P and $1+\frac{(\sigma-1)(\tau-1)}{\max\{1,n-1\}}=3.443$ for RT-D approach.
Again the best performance is given by RT-D which requires knowledge of $\sigma$.
If we do not want to estimate parameter $\sigma$ then we should use FR.
If was shown in \cite{Fercoq:accelerated} that for quadratic objective function FR is always better than RT-P.

\small
\bibliographystyle{plain}
\bibliography{ref2}
\normalsize

%\clearpage

\appendix

\section{Expected Separable Overapproximation}\label{sec:ESO}

\subsection{ Smoothness assumptions}

In this work we assume that the function $f$ is partially separable and smooth, and the purpose of this section is to define these two concepts.
%i.e., let $\calJ$ be a collection of subsets $J \subseteq \{1,\dots,n\}$, then function $f$ can be written as
%\begin{equation}
%\label{eq:structureOfF}
%   f(x) = \sum_{J \in \calJ} f_J (x),
%\end{equation}
%where each function $f_J$ is function only on $\vc{x}{i}, i\in J$. Hence we can also write that
%$$
%f_J(x) = f_J (\vsubset{x}{J}).
%$$
%Let us define a degree of partial separability as follows
%\begin{equation*}
%   \omega \eqdef \max_{J\in\calJ} |J|.
%\end{equation*}
We begin with the definition of partial separability for a smooth convex function, introduced by Richt\'arik and Tak\'a\v{c} in \cite{Richtarik12a} .
\begin{definition}[Partial separability \cite{Richtarik12a}]
A smooth convex function $f:\R^N \to \R$ is  partially separable of degree $\omega$ if there exists a collection ${\cal J}$ of subsets of $\{1,2,\ldots, n\}$ such that
\begin{equation}
\label{E_separability}
     f(x)=\sum_{J\in \mathcal{J}}f_{J}(x) \qquad \text{ and }    \qquad \max_{J\in \mathcal{J}}|J| \leq \omega,
\end{equation}
where for each $J$,  $f_{J}$ is a smooth convex function that depends on $x^{(i)}$ for $i\in J$ only.
\end{definition}
%Clearly $\omega \leq n$.

Now we introduce different types of smoothness assumptions for the function $f$.
Each smoothness type gives rise to a different ESO. Note that \emph{all} of the following smoothness assumptions are \emph{equivalent}. That is, if a given function satisfies one of the assumptions, then there exist constants such that the other assumptions also hold.

The first type of assumption is a classical assumption in the literature \cite{richtarik2012efficient,
Richtarik12,Richtarik12a}.
\begin{assumption}[(Block) Coordinate-wise Lipschitz continuous gradient]
\label{asm:coordinateWiseLipConGrad}
The gradient of $f$ is block Lipschitz, uniformly in $x$, with positive constants $\lip_1,\dots,\lip_n$. That is, for all $x \in \R^N$, $i= 1,\dots,n$ and $h \in \R^{N_i}$ we have
\begin{equation}
\label{S2_Lipschitz}
     \| (\nabla f(x + U_i h))\ii - \gfi \|_{(i)}^* \leq \lip_i \|h\|_{(i)},
\end{equation}
where $\nabla f(x)$ denotes the gradient of $f$ and
\begin{equation}
\label{block_gradient}
     \gfi = U_i^T\nabla f(x) \in \R^{N_i}.
\end{equation}
\end{assumption}
The second type of assumption we make is that each function in the sum \eqref{E_separability} has a Lipschitz continuous gradient.
Such an assumption is made, for example, in \cite{johnson2013accelerating,konevcny2013semi,mahdavi2013mixed,necoara2013distributed}. Moreover, we allow each function to have Lipschitz continuous gradient with a \emph{different constant} (which was also assumed in \cite{necoara2013distributed}).
\begin{assumption}[Lipschitz continuous gradient of sub-functions]
\label{asm:lipGradient}
The gradient of $f_J, J\in \calJ$ has a Lipschitz continuous gradient, uniformly in $x$, with positive constant $\tilde \lip_J$ with respect to some Euclidean norm $\|\cdot\|_{(\tilde J)}$. That is, for all $x \in \R^N$, $J\in\calJ$ and $h \in \R^N$ we have
\begin{equation}\label{eq:lipGradient}
     \| \nabla f_J(x +   h) - \nabla f_J(x) \|_{(\tilde  J)}^* \leq \tilde\lip_J \|h\|_{(\tilde J)}.
\end{equation}
\end{assumption}
Note that this smoothness assumption is more general than that made in \cite{necoara2013distributed} because of the possibility of choosing general norms of the form $\|\cdot\|_{(\tilde J)}$. Further, Assumption \ref{asm:lipGradient} generalizes the smoothness assumptions imposed in \cite{Bradley,richtarik2013distributed}.
%Please note that this smoothness assumption is, due to possibility of choosing norms $\|\cdot\|_{(\tilde J)}$, more general than the one in \cite{necoara2013distributed} and also it generalize smoothness assumptions imposed in \cite{Bradley,richtarik2013distributed}.

The third type of assumption we make is that each function in the sum \eqref{E_separability} has coordinate-wise Lipschitz continuous gradient.
\begin{assumption}[(Block) Coordinate-wise Lipschitz continuous gradient of sub-functions]
\label{asm:lipGradientCoordinateWise}
 The gradient of $f_J, J\in\calJ$ is block Lipschitz, uniformly in $x$, with non-negative constants $\hat\lip_{J,1},\dots,\hat\lip_{J,n}$. That is, for all $x \in \R^N$, $i= 1,\dots,n$, $J\in\calJ$ and $h \in \R^{N_i}$ we have
\begin{equation}\label{eq:lipGradientCoordinateWise}
     \| (\nabla f_J(x + U_i h))\ii -  (\nabla f_J(x))\ii \|_{(i)}^* \leq \hat\lip_{J,i} \|h\|_{(i)}.
\end{equation}
\end{assumption}
One can think of Assumptions \ref{asm:coordinateWiseLipConGrad} and \ref{asm:lipGradient} as being `opposite' to each other in the following sense. If we associate the block coordinates with the columns, and the functions with the rows, we see that Assumption \ref{asm:coordinateWiseLipConGrad} captures the dependence columns-wise, while Assumption \ref{asm:lipGradient} captures the dependence row-wise. Hence, Assumption \ref{asm:lipGradientCoordinateWise} can be thought of as an element-wise smoothness assumption.

To make this more concrete, we present an example that demonstrates how to compute the Lipschitz constants for a quadratic function, under each of the three smoothness assumptions stated above.
\begin{example}\label{exm:quadraticAndLipConstants}
Let the function $f(x) =\frac12 \|Ax-b\|_2^2=\frac12 \sum_{j=1}^m (\vc{b}{i}-\sum_{i=1}^n a_{j,i}  \vc{x}{i})^2$,
where $A \in \R^{m \times n}$ and $a_{j,i}$ is $(j,i)$th element of the matrix $A$.
Let us fix all the norms $\|\cdot\|_{(\tilde J)}$ from Assumption \ref{asm:lipGradient} to be standard Euclidean norms.
Then one can easily verify that equations
\eqref{S2_Lipschitz}, \eqref{eq:lipGradient} and \eqref{eq:lipGradientCoordinateWise}
are satisfied with the following choice of constants
\begin{align*}
 \lip_i &= \sum_{j=1}^m a_{j,i}^2,
 &
 \tilde \lip_j &= \sum_{i=1}^n a_{j,i}^2,
 &
 \hat \lip_{j,i} &= a_{j,i}^2.
\end{align*}
In words, $\lip_i$ is equal to square of the $\ell_2$ norm of $\rm i$th column,
$\tilde \lip_j$ is equal to the square of the $\ell_2$ norm of the $\rm j$th row
and $\hat \lip_{j,i}$ is simply the square of the $\rm{(j,i)}$th element of the matrix $A$.
\end{example}
One could be misled into believing that Assumption \ref{asm:lipGradientCoordinateWise} is the best because it is the most restrictive. However, while this is true for the quadratic objective shown in Example \ref{exm:quadraticAndLipConstants}, for a general convex function, Assumption \ref{asm:lipGradientCoordinateWise} can give Lipschitz constants that lead to worse ESO bounds  (see   Example \ref{exm:logisticLoss} for further details).

%An important consequence of \eqref{S2_Lipschitz} is the following standard inequality \cite[p.57]{Nesterov04}:\begin{equation}\label{S2_upperbound}     f(x+U_i h) \leq f(x) + \langle \gfi, h\rangle+ \frac{\lip_i}{2}\|h\|_{(i)}^2.\end{equation}

\subsection{Expected Separable Overapproximation (ESO)}
\label{sec:asdfoijwofwaefwa}
Now, it is clear that the update $h$ in Algorithm \ref{PCDM}
depends on the ESO parameter  $v$. This shows that the ESO is not just a technical tool; the parameters are actually \emph{used} in Algorithm \ref{PCDM}. Therefore we must be able to obtain/compute these parameters easily. We now present the following three theorems, namely Theorems \ref{thm:niceESO}, \ref{thm:niceNewESO} and \ref{thm:IonESO}, that explain how to obtain the   $v$ parameter for a $\tau$-nice sampling, under different smoothness assumptions.

\begin{theorem}[ESO for a $\tau$-nice sampling, Theorem 14 in \cite{Richtarik12a}]
\label{thm:niceESO}
Let Assumption \ref{asm:coordinateWiseLipConGrad} hold with constants $\lip_1, \dots, \lip_n$ and let $\hatS$ be a $\tau$-nice sampling. Then
 $f:\R^N \to \R$ admits an ESO with respect to the sampling $\hatS$ with parameter
 $$v=\left(1+\frac{(\omega-1)(\tau-1)}{\max\{1,n-1\}}\right) L,
$$
where $L=(\lip_1, \dots, \lip_n)^T$.
\end{theorem}
The obvious disadvantage of Theorem \ref{thm:niceESO} is the fact that $v$ in the ESO, depends on $\omega$. (When $\omega$ is large, so too is $v$.) One can imagine a situation in which $\omega$ is much larger than the average cardinality of $J \in \mathcal{J}$, resulting in a large $v$. For example, if $|J|$ for $J \in \mathcal{J}$ is small for all but one function.

With this in mind, we introduce a new theorem that shows how the ESO in Theorem~\ref{thm:niceESO} can be modified if we know that Assumption  \ref{asm:lipGradientCoordinateWise} holds. In this case, the role of $\omega$ is slightly suppressed.
\begin{theorem}[ESO for a doubly uniform sampling]
\label{thm:NewESOforDUS}
Let Assumption \ref{asm:lipGradientCoordinateWise} hold with constants $\hat \lip_{J,i}, J\in\calJ, i\in \{1,\dots,n\}$ and let $\hatS$ be a doubly uniform sampling. Then
 $f:\R^N \to \R$ admits an ESO with respect to the sampling $\hatS$ with parameter
\begin{equation}
 \bar v = \sum_{J \in \calJ} \left(1+\frac{\left(\frac{\Exp[|\hatS|^2]}{\Exp[|\hatS|]}-1\right)(|J|-1)}{\max\{1,n-1\}}\right)
   (\hat \lip_{J,1},\dots,\hat \lip_{J,n})^T.
\end{equation}
\end{theorem}
\begin{proof}
From Theorem 15 in \cite{Richtarik12a} we know that for each function
$f_J, J\in\calJ$
we have
$$
(f_J, \hatS) \sim
 ESO\left(1+\frac{\left(\frac{\Exp[|\hatS|^2]}{\Exp[|\hatS|]}-1\right)(|J|-1)}{\max\{1,n-1\}},
   (\hat\lip_{J,1},\dots,\hat\lip_{J,n})^T  \right).
$$
Now, using Theorem 10 in \cite{Richtarik12a}, which deals with conic combinations of functions, we have
$$
\left(\sum_{J\in\calJ}f_J, \hatS\right) \sim
 ESO\left(1,
   \sum_{J\in\calJ} \left(1+\frac{\left(\frac{\Exp[|\hatS|^2]}{\Exp[|\hatS|]}-1\right)(|J|-1)}{\max\{1,n-1\}}\right)(\hat\lip_{J,1},\dots,\hat\lip_{J,n})^T  \right).
$$
\end{proof}

The following Theorem is a special case of Theorem \ref{thm:NewESOforDUS}
for a $\tau$-nice sampling.
\begin{theorem}[ESO for a $\tau$-nice sampling, Theorem 1 in \cite{Fercoq:accelerated}]
\label{thm:niceNewESO}
Let Assumption \ref{asm:lipGradientCoordinateWise} hold with constants $\hat \lip_{J,i}, J\in\calJ, i\in \{1,\dots,n\}$ and let $\hatS$ be a $\tau$-nice sampling. Then
$f:\R^N \to \R$ admits an ESO with respect to the sampling $\hatS$ with parameter
\begin{equation}
 \hat v = \sum_{J \in \calJ} \left(1+\frac{(\tau-1)(|J|-1)}{\max\{1,n-1\}}\right)
   (\hat \lip_{J,1},\dots,\hat \lip_{J,n})^T.
\end{equation}
\end{theorem}
\begin{proof}
Notice that, if $\hatS$ is $\tau$-nice sampling, then
 $\Exp[|\hatS|]=\tau$ and $\Exp[|\hatS|^2]=\tau$ and the result follows from Theorem \ref{thm:NewESOforDUS}.
\end{proof}

The following theorem explains how to compute an ESO if Assumption \ref{asm:lipGradient} holds.
This ESO was proposed and proved in \cite{necoara2013distributed}.
\begin{theorem}[ESO for $\tau$-nice sampling, Lemma 1 in \cite{necoara2013distributed}]
\label{thm:IonESO}
Let Assumption \ref{asm:lipGradient} hold with constants $\tilde \lip_J, J\in\calJ$ and let $\hatS$ be a $\tau$-nice sampling. Then
$f:\R^N \to \R$ admits an ESO with respect to the sampling $\hatS$ with parameter
$$
\tilde v = \sum_{J\in\calJ}
  \tilde \lip_J \vsubset{e}{J},
$$
where $e = (1,\dots,1)^T \in \R^n$.
Moreover, this ESO is monotonic.
\end{theorem}

As it is shown in
Theorem \ref{T_convergence_rate}
the speed of the algorithm (number of iterations needed to solve the problem)
depends on ESO parameter $v$ via the
term $\|x_0-x_*\|_v^2$.
 Moreover, for a given objective function and sampling $\hatS$, there may be more than one ESO that could be chosen. Suppose that we have two ESOs to choose from, characterized by two parameters $v_a$ and $v_b$ respectively, and let $v_a < v_b$. In this case, the ESO characterized by $v_a$ will give us (theoretically) faster convergence, and so it is obvious that this ESO should be used. Furthermore, $v_a$ is used as a parameter in Algorithm \ref{PCDM}, and so, intuitively, this faster theoretical convergence, is expected to lead to fast practical performance.

In Section 6.1 in \cite{Fercoq:accelerated} it was shown that for a quadratic objective, the ESO in Theorem \ref{thm:niceESO} is always worse than the ESO from Theorem \ref{thm:niceNewESO}. However, for a general objective the opposite can be true. The following simple example shows that the ESO from Theorems \ref{thm:niceNewESO} and \ref{thm:IonESO} can be $m$ times worse than the ESO from Theorem \ref{thm:niceESO}.
\begin{example}\label{exm:logisticLoss}
Consider the function
$$
f(x)=\sum_{j=1}^m  \underbrace{\log\left(1+e^{-x+\zeta j}\right)}_{f_j(x)},
$$
where $\zeta$ is large.
It is clear that $L_j$ is
$$L_j = \max_{x} (f_j(x))'' = \max_x \frac{e^{x+\zeta j}}{(e^\zeta+e^x)^2}
=\frac14.
$$
Therefore, from  Theorem \ref{thm:IonESO}, we obtain  $\tilde v=\tilde L = \frac{m}{4}$.

On the other hand, Theorem \ref{thm:niceESO} produces an ESO with
$$
v = \max_x  (f(x))'' \approx \frac14,
$$
provided that $\zeta$ is large, e.g. $\zeta=100$.
Hence, in this case, the ESO from Theorem \ref{thm:niceESO} will lead to an algorithm that is approximately $m$ times faster than if the ESOs from Theorems \ref{thm:niceNewESO} or \ref{thm:IonESO} were used.
\end{example}

\emph{Remark.} A thorough discussion of the ESO is presented in \cite[Section~4]{Richtarik12a}. Moreover, \cite[Section~5.5]{Richtarik12a} presents a list of parameters $v$ associated with a particular $f$ and sampling scheme $\hatS$ that give rise to an ESO. Indeed, each of samplings described in Section \ref{S_sampling} in this work gives rise to a   $v$ for which an ESO exists.

%For example, in the fully parallel case (where $n$ blocks are updated at each iteration of PCDM), we have $(\beta,w) = (\omega, L)$, where $\omega$ is the degree of partial separability and $L = (L_1,\dots,L_n)$ is the vector of Lipschitz constants.

\end{document}